\documentclass[11pt,amsfonts]{article}

\usepackage[colorlinks=true, citecolor=red, linkcolor=blue]{hyperref}
\usepackage[utf8]{inputenc}
\usepackage[T1]{fontenc} 
\usepackage{tikz}
\usepackage[english]{babel} 
\usepackage{mathrsfs}
\usepackage{amssymb}
\usepackage[cmex10]{amsmath}
\usepackage{amsthm}
\usepackage{wrapfig}
\usepackage{latexsym}
\usepackage{amsfonts}
\usepackage{color}
\usepackage[font=small, labelfont=bf]{caption}
\usepackage{hyperref} 
\usepackage{enumitem}
\usepackage{ctable}
\usepackage{floatrow}
\usepackage{verbatim}
\usepackage{epstopdf}
\usepackage{float}
\usepackage{appendix}

\usepackage[top=2cm , bottom=2cm , left=2.5cm ,right=2.5cm ]{geometry}

\usepackage{dsfont}
\usepackage{amsmath, amssymb}

\usepackage{cases}

\theoremstyle{plain}
\newtheorem{definition}{Definition}
\newtheorem{theorem}{Theorem}

\newtheorem{proposition}{Proposition}

\newtheorem{lemma}{Lemma}
\newtheorem{remark}{Remark}
\definecolor{assumplabelcolor}{RGB}{0,150,0}

\newtheoremstyle{assumpstyle}
  {3pt}{3pt}{\itshape}{}{\bfseries}{.}{0.5em}
  {\thmname{#1}~{\color{assumplabelcolor}\thmnote{#3}}}

\theoremstyle{assumpstyle}
\newtheorem{assumptionmanual}{Assumption}
\newenvironment{assumption}[1]
  {\renewcommand\theassumptionmanual{#1}\assumptionmanual[#1]}
  {\endassumptionmanual}

\newlist{assumpdesc}{description}{1}
\setlist[assumpdesc]{labelwidth=1cm, leftmargin=!, labelindent=0cm}

\newcommand{\assumptionitem}[2]{%
  \item[\textbf{(#1)}]\label{#2}%
}

\numberwithin{equation}{section}
\numberwithin{theorem}{section}
\numberwithin{proposition}{section}
\numberwithin{definition}{section}
\numberwithin{remark}{section}
\numberwithin{corollary}{section}
\numberwithin{lemma}{section}

\usepackage{graphicx}

\makeatletter
\newcommand*\bigcdot{\mathpalette\bigcdot@{.5}}
\newcommand*\bigcdot@[2]{\mathbin{\vcenter{\hbox{\scalebox{#2}{$\m@th#1\bullet$}}}}}
\makeatother
\begin{document} 
	\sloppy
	
	\begin{center}
        {\Large \textbf{Mean Field Games with Reflected Dynamics: Penalization and Relaxed Control Approach}} \\[0pt]
		~\\[0pt]  Ayoub Laayoun, Badr Missaoui 
        \renewcommand{\thefootnote}{}

        \footnotetext{Mohammed VI Polytechnic University, Rabat, Morocco.}
        \footnotetext{E-mails: \text{ayoub.laayoun@um6p.ma, badr.missaoui@um6p.ma}}
        	\renewcommand{\thefootnote}{\arabic{footnote}}

	\end{center}
	
	\renewcommand{\thefootnote}{\arabic{footnote}}
\begin{abstract}
In this paper, we investigate a class of Mean Field Games (MFGs) in which the state dynamics are governed by multidimensional reflected stochastic differential equations (SDEs). We establish the existence of an equilibrium and show that it can be approximated by the equilibrium of MFGs with non-reflected SDE. This approximation is constructed via a penalization method combined with the relaxed control approach introduced in \cite{laker}. Under a uniform ellipticity condition, and by applying the penalization method together with the mimicking theorem, we prove the existence of a Markovian MFG. Furthermore, under an additional convexity assumption, we demonstrate the existence of a strict-Markovian MFG. In the general case, we prove that relaxed MFG solutions with reflected dynamics can be approximated by strict controls whose dynamics are governed by penalized SDEs.
\end{abstract}
\section{Introduction}
Mean Field Games offer an optimization-based framework to address the complexity of large-population decision problems by approximating the collective behavior of many interacting agents through ideas from statistical physics. Introduced independently by Lasry and Lions \cite{LasryP, lasry2, lasry3} and by Huang et al. \cite{malhame1, malhame2}, MFGs arise as the limiting case of symmetric stochastic differential $N$-player games when $N \to \infty$, where each agent interacts weakly with the population via its empirical distribution. In this limit, the analysis simplifies while still yielding approximate Nash equilibria for the original finite game. The framework relies on the representative agent approach, where the empirical distribution of states is approximated by a deterministic, measure-valued process, leading to an optimization problem for a single player whose optimal strategy must be consistent, in a fixed-point sense, with the collective distribution of all agents.

In this paper, we study a class of MFGs where the state dynamics are governed by a multidimensional reflected SDE. The formulation proceeds as follows:

\begin{equation} \label{R_MFG_dis}
\left\{
\begin{aligned}
&\text{(1) Fix a flow of measures } 
\mu_t: [0, T] \to \mathcal{P}(\mathbb{R}^d) ;\ \\[1ex]
&\text{(2) solve the corresponding stochastic control problem} \\[1ex]
&\quad 
\inf_{\alpha} \mathbb{E} \left(
    \int_0^T f(t, X_t, \mu_t, \alpha_t)\, \mathrm{d}t
    + \int_0^T h(t, X_t, \mu_t)\, \mathrm{d}K_t
    + g(X_T, \mu_T)
\right), \\[1ex]
&\text{subject to the reflected SDE} \\[1ex]
&\quad
\begin{cases}
\mathrm{d}X_t+ \mathrm{d}K_t = b(t, X_t, \mu_t, \alpha_t)\, \mathrm{d}t
+ \sigma(t, X_t, \mu_t, \alpha_t)\, \mathrm{d}B_t
, \\[0.5ex]
\text{Law}(X_0) = \lambda, \quad X_t \in \bar{D}\ \text{a.s.}, \\[0.5ex]
K_t = \displaystyle\int_0^t \eta(X_s)\, \mathrm{d}|K|_s, \quad
|K|_t = \displaystyle\int_0^t \mathbb{1}_{\{X_s \in \partial D\}}\, \mathrm{d}|K|_s < \infty;
\end{cases} \\[2ex]
&\text{(3) solve the fixed point problem }  
\text{Law}(X) = \mu, \text{where X is the optimal state process from 2}.
\end{aligned}
\right.
\end{equation}

A flow $(\mu_t)$ that satisfies the consistency condition in point~3 is referred to as a MFG solution.

The MFG problem has been investigated through three main approaches. The first, proposed by Lasry and Lions~\cite{LasryP}, is an analytical method based on solving a coupled system of forward--backward partial differential equations. The backward equation corresponds to the Hamilton--Jacobi--Bellman equation derived from the individual optimization problem, while the forward equation is a Kolmogorov--Fokker--Planck equation describing the evolution of the population distribution. The second approach is probabilistic, relying on the stochastic Pontryagin maximum principle to formulate the fixed-point problem as a system of McKean--Vlasov forward--backward stochastic differential equations; see, for example,~\cite{probaappr1, probaappr2, delanaly}. The third approach, introduced by Lacker~\cite{laker}, is based on the weak formulation of stochastic control, employing relaxed controls and martingale problems. Using compactness arguments and Berge’s maximum theorem, Lacker established the upper hemi-continuity of the best-response correspondence and applied the Kakutani--Fan--Glicksberg fixed-point theorem to prove the existence of a measure-valued process~$\mu$ consistent with the agents’ optimal responses. More recently, this weak formulation has been extended to the study of MFGs with reflected dynamics; see~\cite{MFG_R_1, MFG_R_j, badr}.

To analyze the MFG system in \eqref{R_MFG_dis}, we consider a penalized version without reflection, denoted by $MFG_n$, defined as follows for each $n \geq 1$:
\begin{equation} \label{MFG_n_dis}
\left\{
\begin{aligned}
&\text{(i) Fix a flow of measures } 
\mu_t^n: [0, T] \to \mathcal{P}(\mathbb{R}^d); \\[1ex]
&\text{(ii) Solve the associated stochastic control problem:} \\[1ex]
&\quad 
\inf_{\alpha} \mathbb{E} \left(
    \int_0^T 
        f(t, X_t^n, \mu_t^n, \alpha_t)
        + n\, h(t, X_t, \mu_t)\, (X_t^n - \pi_{\bar{D}}(X_t^n))
    \, \mathrm{d}t 
    + g(X_T^n, \mu_T^n)
\right); \\[1ex]
&\text{subject to the penalized (non-reflected) dynamics:} \\[1ex]
&\quad
\begin{cases}
\mathrm{d}X_t^n = b(t, X_t^n, \mu_t^n, \alpha_t)\, \mathrm{d}t 
- n\big(X_t^n - \pi_{\bar{D}}(X_t^n)\big)\, \mathrm{d}t 
+ \sigma(t, X_t^n, \mu_t^n, \alpha_t)\, \mathrm{d}B_t, \\[0.5ex]
\text{Law}(X_0^n) = \lambda;
\end{cases} \\[2ex]
&\text{(iii) Solve the fixed-point problem: } 
\text{Law}(X_t^n) = \mu_t^n, \quad \forall\, t \in [0, T].
\end{aligned}
\right.
\end{equation}

When there is no control $ \alpha $ and the coefficients $ b $ and $ \sigma $ do not depend on the measure flow $ \mu $, it is well known that the penalized processes $ X^n $ in point (ii) of \eqref{MFG_n_dis} converge in law to the weak solution $ X $ of the reflected SDE described in point 2 of \eqref{R_MFG_dis}. Under Lipschitz conditions on $ b $ and $ \sigma $, Lions et al.~\cite{Lion-no-cont} showed that
\begin{equation*}
\mathbb{E}\Big( \sup_{t \in [0,T]} |X_t^n - X_t|^2 \Big) \to 0.
\end{equation*}
Bahlali et al.~\cite{bahlali-PDE} later proved that when $ b $ and $\sigma$ are only continuous, the pair
\begin{equation*}
\left(X^n,\; K^n := \int_0^\cdot n\big(X_s^n - \pi_{\bar{D}}(X_s^n)\big)\,\mathrm{d}s\right)
\end{equation*}
still converges in law to $ (X, K) $ with respect to the uniform topology.

For the controlled setting, without the interaction term $\mu$, the penalized boundary term $h$, and the terminal cost $g$, Menaldi~\cite{menaldi} proved that
\begin{equation*}
\mathbb{E}\Big( \sup_{t \in [0,T]} |X_t^n - X_t|^p \Big) \to 0, \qquad 1 \le p < \infty,
\end{equation*}
where $ X^n $ is the optimal state for the penalized control problem in point (ii) of \eqref{MFG_n_dis} and $ X $ is the optimal state for the reflected problem in point 2 of \eqref{R_MFG_dis}. The convergence holds uniformly in the control $ \alpha $. He also showed that the optimal cost for the penalized problem converges to that of the reflected control problem.

In our work, we consider a MFG involving both control and mean field interactions, where the dynamics are reflected on the boundary of a convex open subset $D$ of $\mathbb{R}^d$. We extend the penalization ideas from \cite{Lion-no-cont, menaldi} to build a solution to \eqref{R_MFG_dis}. For each $ n $, we use Lacker’s method \cite{laker} to prove the existence of an $MFG_n$ solution $ \mu^n $ to the penalized problem \eqref{MFG_n_dis}. From these solutions, we construct a solution to the reflected problem \eqref{R_MFG_dis} and show that $ \mu^n \to \mu $, where $\mu $ is an MFG solution of \eqref{R_MFG_dis}. The argument relies on tools such as tightness and convergence of the associated processes.

This paper is organized as follows. In Section~\ref{sec-main-result}, we recall the notion of the canonical space and introduce relaxed controls for stochastic control problems governed by systems of reflected SDEs and their penalized counterparts. In the same section, we present the formulation of MFGs for both the reflected and penalized SDE frameworks, state the assumptions on the model coefficients, and establish our first main result (Theorem~\ref{MFGR}). 
Section~\ref{section-proof} is devoted to the proof of this theorem. 
Section~\ref{sec-mark} presents our second main result: the existence of Markovian and strict Markovian MFG solutions under a convexity assumption (Theorem~\ref{theorem:Markovian}). 
Finally, in Section~\ref{sec-appro}, we show that, in the general setting without the convexity assumption, a relaxed MFG solution with reflected dynamics can be approximated by strict controls with penalized SDE dynamics. This constitutes our third main result (Theorem~\ref{stric-Non refle-equi relaxed}).

\section{Assumptions and Existence of MFG Solutions}\label{sec-main-result}

Throughout this paper, for any set $A \subset \mathbb{R}^d$, we denote by $\mathcal{C}^d(\bar{A})$ the space of continuous functions from $[0,T]$ to $\bar{A}$, and by $\mathcal{A}^d(\bar{A}) \subset \mathcal{C}^d(\bar{A})$ the subset of functions with bounded variation. Both spaces are endowed with the topology of uniform convergence. In addition, we write $\mathcal{C}^d := \mathcal{C}^d(\mathbb{R}^d)$. 

Let $(\mathcal{Z}, \rho)$ be a metric space. We denote by $\mathcal{P}(\mathcal{Z})$ the set of all probability measures on the measurable space $(\mathcal{Z}, \mathcal{B}(\mathcal{Z}))$, where $\mathcal{B}(\mathcal{Z})$ stands for the Borel $\sigma$-field of $\mathcal{Z}$. The space $\mathcal{P}(\mathcal{Z})$ is equipped with the topology of weak convergence of measures.  \\
For any $p \ge 1$, define
\begin{equation*}
\mathcal{P}_p(\mathcal{Z}) := \left\{ \mathbb{P} \in \mathcal{P}(\mathcal{Z}) : \int_{\mathcal{Z}} \rho(x, x_0)^p \, \mathbb{P}(dx) < \infty \text{ for some (hence all) } x_0 \in \mathcal{Z} \right\}.
\end{equation*}
The space $\mathcal{P}_p(\mathcal{Z})$ is endowed with the $p$-Wasserstein metric
\begin{equation*}
W_{\mathcal{Z}, p}(\mu, \nu)
= \inf_{\pi \in \Pi(\mu, \nu)}
\left( \int_{\mathcal{Z} \times \mathcal{Z}} \rho(x, y)^p \, \pi(dx, dy) \right)^{1/p},
\quad \mu, \nu \in \mathcal{P}_p(\mathcal{Z}),
\end{equation*}
where
\begin{equation*}
\Pi(\mu, \nu) := \{ \pi \in \mathcal{P}(\mathcal{Z} \times \mathcal{Z}) : \pi \text{ has marginals } \mu \text{ and } \nu \}.
\end{equation*}

We denote by $\mathcal{U}$ the set of all probability measures on $[0,T] \times U$ whose first marginal coincides with the Lebesgue measure on $[0,T]$ and whose second marginal is a probability measure on $U$. The space $\mathcal{U}$ is equipped with a metric
\begin{equation*}
d_{\mathcal{U}}(q_1, q_2)
:= W_{[0,T] \times U, 2}\left(\frac{q_1}{T}, \frac{q_2}{T}\right),
\end{equation*}
where $W_{[0,T] \times U, 2}$ denotes the 2-Wasserstein distance.  
Since $U$ is compact under assumption~\hyperlink{A.6}{(A.6)}, it follows that the metric space $(\mathcal{U}, d_{\mathcal{U}})$ is compact and therefore complete and separable.

Let $(X, K)$ be a solution of the reflected SDE
\begin{equation} 
\label{refl-SDE-sim}
\begin{cases}
\mathrm{d}X_t + \mathrm{d}K_t
= B(t, X_t)\, \mathrm{d}t
+ A(t, X_t)\, \mathrm{d}t
, \\[1.1em]
\mathrm{Law}(X_0) = \lambda, \quad X_t \in \bar{D}, \ \forall\, t \ge 0, \\[0.4em]
|K|_t = \displaystyle\int_0^t \mathbf{1}_{\{ X_s \in \partial D \}}\, \mathrm{d}|K|_s < \infty, 
\quad 
K_t = \displaystyle\int_0^t \eta(X_s)\, \mathrm{d}|K|_s,
\end{cases}
\end{equation}
where $\eta(x)$ denotes the unit inward normal to $\partial D$ at $x \in \partial D$, and $K$ is a bounded variation process that ensures $X_t \in \bar{D}$.\\
For simplicity, we will often use the compact notation
\begin{equation*}
\mathrm{d}X_t+ \mathrm{d}K_t 
= B(t, X_t)\, \mathrm{d}t
+ A(t, X_t)\, \mathrm{d}t
,
\end{equation*}
with the reflection conditions on $K$ implicitly understood.

We introduce the following notation for any set $A \subset \mathbb{R}^d$:

\begin{itemize}
    \item \textbf{Canonical space:}
    \begin{equation*}
        \Omega[\bar{A}] := \mathcal{C}^d(\bar{A}) \times \mathcal{U} \times \mathcal{A}^d(\bar{A}).
    \end{equation*}

    \item \textbf{$\sigma$-algebra:}
    \begin{equation*}
        \mathcal{F}[\bar{A}] := \mathcal{F}^{\mathcal{C}^d(\bar{A})} \times \mathcal{F}^{\mathcal{U}} \times \mathcal{F}^{\mathcal{C}^d(\bar{A})},
    \end{equation*}
    where $\mathcal{F}^{\mathcal{C}^d(\bar{A})}$ and $\mathcal{F}^{\mathcal{U}}$ denote the Borel $\sigma$-algebras of $\mathcal{C}^d(\bar{A})$ and $\mathcal{U}$, respectively.

    \item \textbf{Coordinate projections $X$, $Q$, and $K$:} For each $\omega = (x, q, k) \in \Omega[\bar{A}]$,
    \begin{equation*}
        X(\omega) = x, \quad Q(\omega) = q, \quad K(\omega) = k.
    \end{equation*}
    According to \cite[Lemma 3.2]{laker}, the measure $Q$ admits a predictable disintegration:
    \begin{equation*}
        Q(\mathrm{d}t, \mathrm{d}u) = Q_t(\mathrm{d}u) \, \mathrm{d}t.
    \end{equation*}

    \item \textbf{Filtration:} $\mathbb{F} := (\mathcal{F}_t)_{t \in [0,T]}$ with
    \begin{equation*}
        \mathcal{F}_t := \mathcal{F}_t^X \otimes \mathcal{F}_t^Q \otimes \mathcal{F}_t^K,
    \end{equation*}
    where
    \begin{align*}
        \mathcal{F}_t^X &= \sigma(X_s : s \leq t), \\
        \mathcal{F}_t^Q &= \sigma(Q(F) : F \in \mathcal{B}([0,t] \times U)), \\
        \mathcal{F}_t^K &= \sigma(K_s : s \leq t).
    \end{align*}
\end{itemize}
\begin{assumption}{\textbf{(A)}}
\label{ass:A}
\begin{assumpdesc}
\assumptionitem{A.1}{A.1}
\label{A.1}
The functions $ b, \sigma, f, $ and $ g $ of $(t, x, \mu, u)$ are measurable in $ t $ and continuous in  $(x, \mu, u) $.
\assumptionitem{A.2}{A.2}
\label{A.2}
The function $h$ is continuous in $(t, x, \mu)$.
\assumptionitem{A.3}{A.3}
\label{A.3}
There exist constants $C_1, C_2 > 0$ such that, for all $(t, u) \in [0, T] \times U$, $x, y \in \mathbb{R}^d$, and $\mu, \nu \in \mathcal{P}_2(\mathbb{R}^d)$,
\begin{equation*}
|b(t, x, \mu, u) - b(t, y, \nu, u)| + |\sigma(t, x, \mu, u) - \sigma(t, y, \nu, u)| 
\leq C_1\big( |x - y| + W_{\mathbb{R}^{d},2}(\mu, \nu) \big),
\end{equation*}
\begin{equation*}
|b(t, x, \mu, u)| \leq C_2\Big(1+|x|+\Big(\int_{\mathbb{R}^d} |z|^2 \mu(\mathrm{d}z)\Big)^{1/2}\Big),
\end{equation*}
and
\begin{equation*}
 |\sigma\sigma^\top(t, x, \mu, u)| \leq C_2\Big(1+|x|^2+\int_{\mathbb{R}^d} |z|^2 \mu(\mathrm{d}z)\Big).
\end{equation*}
Here, $\sigma^\top$ denotes the transpose of $\sigma$.
\assumptionitem{A.4}{A.4}
\label{A.4}
There exist strictly positive constants $C_3$ and $C_4$ such that, for all $(t, x, \mu, u)$,
\begin{equation*}
|g(x, \mu)| \leq C_4 \left( 1 + |x|^2 + \int_{\mathbb{R}^d} |z|^2 \mu(\mathrm{d}z) \right),
\end{equation*}
\begin{equation*}
|h(t, x, \mu)| \leq C_4\left(1+|x|+\left(\int_{\mathbb{R}^d} |z|^2 \mu(\mathrm{d}z)\right)^{1/2}\right),
\end{equation*}
and
\begin{equation*}
|f(t, x, \mu, u)| \leq C_4 \left( 1 + |x|^2 + \int_{\mathbb{R}^d} |z|^2 \mu(\mathrm{d}z) \right).
\end{equation*}
\assumptionitem{A.5}{A.5}
\label{A.5}
The initial law $\lambda$ belongs to $\mathcal{P}_{q'}(\bar{D})$ for some $q' > 2$.
\assumptionitem{A.6}{A.6}
\label{A.6}
The control space $ U $ is compact subset $U\subset\mathbb{R}^N$, $N\geq 1$.
\assumptionitem{A.7}{A.7}
\label{A.7}
The state space $D$ is a convex open subset of $\mathbb{R}^d$ such that $0 \in D$.
\end{assumpdesc}
\end{assumption}
\begin{remark}
For $\Psi \in \{\sigma \sigma^\top, f \}$ and corresponding constants $C' \in \{ C_2, C_4 \}$, the growth bounds in  Assumptions \hyperlink{A.3}{(A.3)} and \hyperlink{A.4}{(A.4)} can be relaxed to allow an additional dependence on the control variable $u$ as follows:
\begin{align*}
&|\Psi(t, x, \mu, u)| \leq C' \left( 1 + |x|^2 + \int_{\mathbb{R}^d} |z|^2 \, \mu(\mathrm{d}z) + \phi(u) \right) \\
&|b(t, x, \mu, u)|\leq C_2 \left( 1 + |x|+ \left(\int_{\mathbb{R}^d} |z|^2 \, \mu(\mathrm{d}z)\right)^{1/2} + \phi(u) \right)
\end{align*}
for some continuous function $(\phi^1, \phi^2, \phi^3) \ni\phi : U \to \mathbb{R}_+$.  
However, since $U$ is compact by Assumption \hyperlink{A.6}{(A.6)}, the function $\phi$ is bounded on $U$. Consequently, the bounds in Assumptions \hyperlink{A.3}{(A.3)} and \hyperlink{A.4}{(A.4)} can be taken to hold uniformly in $u$. 
\end{remark}
\subsection{Relaxed Formulation of the MFG with Penalized Dynamics}

Following the relaxed control framework developed in \cite{laker}, we define, for each $ n \geq 1 $, the relaxed control associated with the penalized dynamics. The corresponding control problem for the MFG with penalized dynamics is given by
\begin{equation} 
\begin{cases}
\displaystyle \inf_{\alpha} \mathbb{E} \left( \int_0^T f^n(t, X_t^n, \mu_t^n, \alpha_t) \, \mathrm{d}t + g(X_T^n, \mu_T^n) \right), \\
\text{subject to:} \\
\mathrm{d}X_t^n = b^n(t, X_t^n, \mu_t^n, \alpha_t) \, \mathrm{d}t + \sigma(t, X_t^n, \mu_t^n, \alpha_t) \, \mathrm{d}B_t, \\
\text{Law}(X_0^n) = \lambda,
\end{cases}
\end{equation}
where the penalized coefficients are defined by
\begin{equation*}
b^n(t, x, \mu, \alpha) := b(t, x, \mu, \alpha) - n(x - \pi_{\bar{D}}(x)), \quad f^n(t,x,\mu,\alpha) := f(t,x,\mu,\alpha) + n h(t,x, \mu)(x - \pi_{\bar{D}}(x)).
\end{equation*}
\begin{remark}\label{Penalized-coeff}
For each $ n \geq 1 $, the penalized drift $ b^n $ is Lipschitz continuous and satisfies the same growth condition as in Assumption \hyperlink{A.3}{(A.3)}, with a constant $ C^1_n $. Likewise, the penalized running cost $ f^n $ is continuous in $ (x, \mu, u) $ and satisfies the same condition as in Assumption \hyperlink{A.4}{(A.4)}, with a constant $ C^2_n $.
\end{remark}

\begin{definition} \label{penalized-relaxed} 
For each $ n \geq 1 $ and $ \mu^n \in \mathcal{P}_2(\mathcal{C}^d) $, we define the set of admissible relaxed controls:
\begin{equation*}
\mathcal{R}_n(\mu^n) := \left\{ P^n \in \mathcal{P}(\Omega[\mathbb{R}^d]) \,:\, \text{conditions (i)–(ii) hold} \right\},
\end{equation*}
where:
\begin{description}
\item[i)] $ P^n \circ X_0^{-1} = \lambda $;
\item[ii)] For every $ \phi \in \mathcal{C}_b^2(\mathbb{R}^d) $, the process $ \mathcal{M}^{\mu, \phi, n} $ defined by
\begin{equation*}
\mathcal{M}^{\mu, \phi, n}_t := \phi(X_t) - \int_0^t \int_U \mathcal{L}^n \phi(s, X_s, \mu^n_s, u) Q_s(\mathrm{d}u) \mathrm{d}s
\end{equation*}
is an $(\mathbb{F}, P^n)$-martingale, with the operator
\begin{equation*}
\mathcal{L}^n \phi(t, x, \mu, u) = b^n(t, x, \mu, u)^\top D\phi(x) + \frac{1}{2} \operatorname{Tr} \left( \sigma \sigma^\top(t, x, \mu, u) D^2\phi(x) \right).
\end{equation*}
\end{description}
\end{definition}
The cost functional associated with $ P^n \in \mathcal{R}_n(\mu^n) $ is defined as
\begin{equation*}
J^n(\mu^n, P^n) := \mathbb{E}^{P^n} \left( \int_0^T \int_U f^n(t, X_t, \mu^n_t, u) Q_t(\mathrm{d}u) \mathrm{d}t + g(X_T, \mu^n_T) \right).
\end{equation*}

We also define the (possibly empty) set of optimal admissible laws
\begin{equation*}
\mathcal{R}_n^*(\mu^n) := \arg \min_{P^n \in \mathcal{R}_n(\mu^n)} J^n(\mu, P^n).
\end{equation*}

By Proposition~3.5 in \cite{laker}, for each $ n \geq 1 $  and corresponding measure $  \mu^n $ , 
the set $  \mathcal{R}_n(\mu^n) $  coincides with the collection of laws of weak solutions 
to a SDE driven by martingale measures. 
We adopt this representation throughout our framework.

\begin{proposition}\label{rep martingale for nMFG}
A probability measure $  P^n \in \mathcal{R}_n(\mu^n) $  if and only if there exists a filtered probability space $  (\Omega^n, \mathcal{F}^n_t, \mathbb{Q}^n) $  supporting:
\begin{itemize}
    \item a $  d $ -dimensional $  \mathcal{F}^n_t $ -adapted process $  X^n $ ,
    \item a collection $  M^n = (M^n_1, \ldots, M^n_m) $  of $  m $  orthogonal martingale measures on $ U \times [0, T] $ , with intensity $  Q^n_t(\mathrm{d}a)\,\mathrm{d}t $ ,
\end{itemize}
such that $  \mathbb{Q}^n \circ (X^n, Q^n)^{-1} = P^n $ , and the following state equation is satisfied:
\begin{equation*}
\mathrm{d}X_t^n = \int_U b^n(t, X_t^n, \mu^n_t, u)\, Q_t^n(\mathrm{d}u)\, \mathrm{d}t + \int_U \sigma(t, X_t^n, \mu^n_t, u)\, M^n(\mathrm{d}u, \mathrm{d}t).
\end{equation*}
Under standard Lipschitz and growth assumptions, this equation admits a unique strong solution on any such filtered probability space.
\end{proposition}

\begin{definition}
For each $  n \geq 1$, a relaxed MFG solution is a probability measure $  P^n \in \mathcal{R}_n^*(P^n \circ X^{-1}) $ .

If, in addition, $  P^n $ -almost surely
\begin{equation*}
Q(\mathrm{d}t, \mathrm{d}u) = \delta_{u_t}(\mathrm{d}u)\mathrm{d}t
\end{equation*}
for some progressively measurable process $  u $ , then $  P^n $  is called a strict solution.
\end{definition}

The following theorem is a direct consequence of Theorem~4.1 in \cite{laker}. 
For each $n \geq 1$, the coefficients $(b^n, \sigma, f^n, g)$ satisfy Assumption~\textbf{(A)} of \cite{laker} (with $p = p_{\sigma} = 2$ and $p' = q'$), as ensured by Remark~\ref{Penalized-coeff} together with Assumption~\ref{ass:A}.

\begin{theorem}\label{MFG-existence for penalization}
For each $n \geq 1$, there exists a probability measure $P^n$ satisfying
\begin{equation*}
P^n \in \mathcal{R}_n^*(\mu^n),
\qquad 
\mu^n = P^n \circ X^{-1}.
\end{equation*}
\end{theorem}
\subsection{Relaxed MFGs with Reflected Dynamics}
We introduce the concept of relaxed controls, following the framework used for controlled SDEs without reflection. To formulate the martingale problem for our controlled system, we adopt the method of Ikeda and Watanabe \cite{Ikeda}. The existence of optimal relaxed controls for reflected SDEs has been established in \cite{Kushner-Variance,Laayoun1} using the compactification method. We employ a similar approach to address MFGs with reflected dynamics. The notion of relaxed controls used here is based on \cite{Kushner-Variance}.
\begin{definition} \label{relaxed}
Let $\mu$ be a given probability measure in $\mathcal{P}_2(\mathcal{C}^d(\bar{D}))$.  
We denote by $\mathcal{R}_{\mathrm{ref}(D)}(\mu)$ the set of all probability measures 
$P \in \mathcal{P}(\Omega[\bar{D}])$ satisfying the following conditions:
\begin{description}
    \item[1)] $P \circ X_0^{-1} = \lambda$.
    \item[2)] For each $\phi \in \mathcal{C}^2_b(\mathbb{R}^d)$, the process 
    $\mathcal{M}^{\mu, \phi}$ defined by
    \begin{equation} \label{equ02.5}
        \mathcal{M}^{\mu, \phi}_t 
        := \phi(X_t) 
        - \int_0^t \int_U \mathcal{L}\phi(s, X_s, \mu_s, u)\, Q_s(\mathrm{d}u)\, \mathrm{d}s
        + \int_0^t D \phi(X_s)\, \mathrm{d}K_s
    \end{equation}
    is an $(\mathbb{F}, P)$-martingale, where the operator $\mathcal{L}$ is defined by
    \begin{equation*}
        \mathcal{L}\phi(t, x, \mu, u)
        := b(t, x, \mu, u)^\top D\phi(x)
        + \frac{1}{2}\, \operatorname{Tr}\!\left( 
        \sigma\sigma^\top(t, x, \mu, u)\, D^2\phi(x)
        \right).
    \end{equation*}
    
    \item[3)] $P$-a.s., for every $t \in [0,T]$,
    \begin{equation*}
    |K|_t = \int_0^t \mathbf{1}_{\{ X_s \in \partial D \}}\, \mathrm{d}|K|_s < \infty,
    \qquad 
    K_t = \int_0^t \eta(X_s)\, \mathrm{d}|K|_s.
    \end{equation*}
\end{description}
\end{definition}

\begin{remark}
Under Assumption \hyperlink{A.3}{(A.3)}, the set $\mathcal{R}_{\mathrm{ref}(D)}(\mu)$ is nonempty for every $\mu \in \mathcal{P}_2(\mathcal{C}^d(\bar{D}))$. 
Indeed, fixing a constant control $\alpha_0$, it follows from \cite[Theorem~1.1]{Sznitman} (see also \cite[Theorem~4.1]{Tanaka-existence}) that there exists a probability measure 
$P \in \mathcal{R}_{\mathrm{ref}(D)}(\mu)$ such that
\begin{equation*}
P\left( Q_t = \delta_{\alpha_0} \text{ for almost every } t \in [0,T] \right) = 1.
\end{equation*}
\end{remark}
The cost functional associated with an admissible law $P \in \mathcal{R}_{\mathrm{ref}(D)}(\mu)$ is defined by
\begin{equation} \label{eq:cost-functional}
J(\mu, P)
:= \mathbb{E}^{P}\left(
\int_0^T \int_U f(t, X_t, \mu_t, u) \, Q_t(\mathrm{d}u)\, \mathrm{d}t
+ \int_0^T h(t, X_t, \mu_t) \, \mathrm{d}K_t
+ g(X_T, \mu_T)
\right),
\end{equation}
which is well-defined by Assumption~\hyperlink{A.4}{(A.4)}.

\noindent We then define the (possibly empty) set of optimal admissible laws by
\begin{equation*}
\mathcal{R}_{\mathrm{ref}(D)}^*(\mu)
:= \arg\min_{P \in \mathcal{R}_{\mathrm{ref}(D)}(\mu)} J(\mu, P).
\end{equation*}
The following proposition provides a characterization of the set $\mathcal{R}_{\mathrm{ref}(D)}(\mu)$ in terms of reflected SDEs driven by martingale measures. Since each element of $\mathcal{R}_{\mathrm{ref}(D)}(\mu)$ solves the martingale problem described in point~2) of Definition~\ref{relaxed}, possibly after an enlargement of the underlying probability space, El Karoui and Méléard~\cite{Mel} proved that, in the non-reflected case, the set $\mathcal{R}_{\mathrm{ref}(\mathbb{R}^d)}(\mu)$ coincides with the collection of solutions to an SDE driven by a martingale measure. Analogous representations also hold in the presence of reflection; see, for instance,~\cite{kushnerbook}.
\begin{proposition} \label{rep martingale for Reflected dynamics}
Let $\mu \in \mathcal{P}(\mathcal{C}^{d}(\bar D))$. 
Then $P \in \mathcal{R}_{\mathrm{ref}(D)}(\mu)$ if and only if there exists a filtered probability space 
$(\Omega', \mathcal{F}'_t, \mathbb{Q}')$
supporting a $d$-dimensional $\mathcal{F}'_t$-adapted process $(X', K')$, 
and $m$ orthogonal $\mathcal{F}'_t$-martingale measures 
$M' = (M'_1, \ldots, M'_m)$ on $U \times [0, T]$ with intensity $Q'_t(\mathrm{d}u)\, \mathrm{d}t$, such that
\begin{equation*}
\mathbb{Q}' \circ (X', K', Q')^{-1} = P,
\end{equation*}
and the reflected dynamics satisfy
\begin{equation} \label{rep-eq1}
\mathrm{d}X'_t+ \mathrm{d}K'_t 
= \displaystyle\int_U b(t, X'_t, \hat{\mu}_t, u)\, Q'_t(\mathrm{d}u)\, \mathrm{d}t
+ \displaystyle\int_U \sigma(t, X'_t, \hat{\mu}_t, u)\, M'(\mathrm{d}u, \mathrm{d}t)
. 
\end{equation}
\end{proposition}
Under the standard Lipschitz continuity and linear growth conditions on $b$ and $\sigma$ with respect to the state variable $x$, namely Assumption~\hyperlink{A.3}{(A.3)}, the reflected SDE~\eqref{rep-eq1} admits a strong solution; see, for example,~\cite{Ishii}.

The following provides the formal definition of a MFG solution with reflected dynamics.

\begin{definition}
A relaxed MFG solution is a probability measure 
\begin{equation*}
P \in \mathcal{P}_2(\Omega[\bar{D}])
\end{equation*}
such that
\begin{equation*}
P \in \mathcal{R}_{\mathrm{ref}(D)}^{*}\bigl(P \circ X^{-1}\bigr),
\end{equation*}
that is, $P$ is a fixed point of the correspondence
\begin{equation*}
\mu \longmapsto 
\bigl\{\, P \circ X^{-1} : P \in \mathcal{R}_{\mathrm{ref}(D)}^{*}(\mu)\, \bigr\},
\end{equation*}
we also refer to the induced flow of measures $\mu = P \circ X^{-1}$ as a relaxed MFG solution.

A \emph{strict MFG solution} is a relaxed solution $P$ for which there exists a progressively measurable control process $u$ such that
\begin{equation*}
P\bigl(Q(\mathrm{d}t,\mathrm{d}u) = \delta_{u_t}(\mathrm{d}u)\,\mathrm{d}t\bigr) = 1.
\end{equation*}

A relaxed solution $P$ is called \emph{Markovian} if its control measure $Q$ satisfies
\begin{equation*}
P\bigl(Q(\mathrm{d}t,\mathrm{d}u)
= \mathrm{d}t\, q(t,X_t)(\mathrm{d}u)\bigr) = 1
\end{equation*}
for some measurable mapping 
$q : [0,T] \times \mathbb{R}^d \to \mathcal{P}(U)$.

It is called \emph{strict Markovian} if there exists a measurable feedback control 
$\alpha : [0,T] \times \mathbb{R}^d \to U$ such that
\begin{equation*}
P\bigl(Q(\mathrm{d}t,\mathrm{d}u)
= \mathrm{d}t\, \delta_{\alpha(t,X_t)}(\mathrm{d}u)\bigr) = 1.
\end{equation*}
\end{definition}

We begin with the first main result of this paper, stated in the following theorem. 
Its proof is given in Section~\ref{section-proof}.
\begin{theorem}\label{MFGR}
Under Assumption~\ref{ass:A}, there exists a relaxed MFG solution with reflected dynamics.

Moreover, let $P$ be a MFG solution with reflected dynamics, and let $P^n$ be a MFG solution with penalized dynamics given by Theorem~\ref{MFG-existence for penalization}. Denote by $\mu = P \circ X^{-1}$ and $\mu^n = P^n \circ X^{-1}$. Then the following properties hold:
\begin{description}
    \item[1)] $\mu^n \to \mu$ in $\mathcal{P}_{2}(\mathcal{C}^d)$,
    \item[2)] $J^n(\mu^n, P^n) \to J(\mu, P)$.
\end{description}
\end{theorem}
\section{Proof of Theorem \ref{MFGR}}\label{section-proof}
By Theorem~\ref{MFG-existence for penalization}, for each $n \in \mathbb{N}^*$ there exists 
$P^n \in \mathcal{R}_n^*(\mu^n)$ such that $\mu^n = P^n \circ X^{-1}$. 
According to Proposition~\ref{rep martingale for nMFG}, one can construct a filtered probability space 
$(\Omega^n, (\mathcal{F}^n_t)_{t\ge0}, \mathbb{Q}^n)$ supporting a $d$-dimensional process $X^n$ and 
$m$ orthogonal martingale measures $M^n = (M^n_1, \dots, M^n_m)$ with intensity 
$Q^n_t(\mathrm{d}u)\mathrm{d}t$, such that
\begin{equation*}
\mathbb{Q}^n \circ (X^n, Q^n)^{-1} = P^n, 
\qquad 
\mu^n = \mathbb{Q}^n \circ (X^n)^{-1},
\end{equation*}
and the process $X^n$ satisfies
\begin{equation*}
\begin{cases}
\mathrm{d}X^n_t
= \displaystyle\int_U b(t,X^n_t,\mu^n_t,u)\,Q^n_t(\mathrm{d}u)\,\mathrm{d}t
- n\big(X^n_t - \pi_{\bar D}(X^n_t)\big)\,\mathrm{d}t
+ \displaystyle\int_U \sigma(t,X^n_t,\mu^n_t,u)\,M^n(\mathrm{d}u,\mathrm{d}t),\\
\mathbb{Q}^n \circ (X^n_0)^{-1} = \lambda .
\end{cases}
\end{equation*}
Define
\begin{equation*}
K^n_t := \int_0^t n\big(X^n_s - \pi_{\bar D}(X^n_s)\big)\,\mathrm{d}s .
\end{equation*}
Then the above dynamics can be written in the compact form
\begin{equation}\label{equation1}
\mathrm{d}X^n_t + \mathrm{d}K^n_t
= \int_U b(t,X^n_t,\mu^n_t,u)\,Q^n_t(\mathrm{d}u)\,\mathrm{d}t
+ \int_U \sigma(t,X^n_t,\mu^n_t,u)\,M^n(\mathrm{d}u,\mathrm{d}t).
\end{equation}
\paragraph*{Strategy of the Proof.}
The proof proceeds by first establishing the relative compactness of the sequence of probability measures $(P^n)_{n \ge 1}$ and then showing that any limit point of this sequence corresponds to a MFG solution with reflected dynamics. To this end, the argument is divided into two main steps:
\begin{itemize}
    \item In Subsection~\ref{subsection 4.1}, we prove the tightness of the sequence of processes $(X^n, K^n, Q^n, M^n)_{n \geq 1}$.
    \item In Subsection~\ref{subsection 4.2}, we identify the limit and prove Theorem~\ref{MFGR}.
\end{itemize}

To prove the admissibility of the limit of $P^n$, we follow the approach of \cite{Lauka,lauka1}. 
We begin by establishing the tightness of the sequence $(X^n, K^n, Q^n, M^n)$. 
In particular, the pair $(X^n, K^n)$ is studied with respect to the S-topology on the space 
$\mathcal{D}(\mathbb{R}^+, \mathbb{R}^d)^2$ of càdlàg $\mathbb{R}^d$-valued functions, 
introduced by Jakubowski in \cite{jakub}. 
Since the coefficients $b$ and $\sigma$ are Lipschitz continuous, the reflected SDE admits a unique solution. 
Then, by applying Theorem~4.3 in \cite{Lauka}, we deduce that the limit of $P^n$ is admissible.

\subsection{Relative Compactness of the Penalized Processes $(X^n, K^n, Q^n, M^n)$}\label{subsection 4.1}
The following lemmas establish the relative compactness of the processes $(X^n, K^n, Q^n, M^n)$.

\begin{lemma} \label{lemma1}
Assume \hyperlink{A.3}{(A.3)}. For every $ q \ge 1 $ with $2q \in [2, q'] $, there exists a constant $ C > 0 $, depending only on $ q $, $ |\lambda|^{q'} $, $ T $, and the constant $ C_2 $ in \hyperlink{A.3}{(A.3)}, such that
\begin{equation} \label{eq3.1}
\sup_n \mathbb{E}^{\mathbb{Q}^n} \left( \sup_{0 \leq t \leq T} |X^n_t|^{2q} + \sup_{0 \leq t \leq T} |K^n_t|^{2q} + |K^n|^q_{[0,T]} \right) \leq C.
\end{equation}
\end{lemma}
\begin{proof}
Applying Itô’s formula to $ |X^n_t|^2 $, we obtain the following result
\begin{align*}
|X^n_t|^2+2\displaystyle\int_{0}^{t}\big\langle X^n_s,\mathrm{d}K^n_s \big\rangle &= |X_{0}|^2 +2\displaystyle\int_0^t\int_{U}^{}\big\langle X^n_s, b^n(s,X^n_s,\mu^n_s,u)Q^n_s(\mathrm{d}u)\mathrm{d}s\big\rangle\\
&+2 \int_0^t\int_{U}^{} \big\langle X^n_s,\sigma(s,X^n_s, \mu^n_s, u)M^n(\mathrm{d}u,\mathrm{d}s)\big\rangle 
+\int_{0}^{t} \int_{U}^{}\big|\sigma(s,X^n_s, \mu^n_s, u)\big|^2Q^n_s(\mathrm{d}u)\mathrm{s}
\end{align*}
Therefore, there exists a constant $ C$, which may vary from line to line, such that
\begin{align*}
\left( |X^n_t|^2 + 2\int_0^t \langle X^n_s, \mathrm{d}K^n_s \rangle \right)^q 
&\leq C \Bigg( |X_0|^{2q}+
\left| \int_0^t \int_U \langle X^n_s, b^n(s, X^n_s, \mu^n_s, u) \rangle Q^n_s(\mathrm{d}u) \, \mathrm{d}s \right|^q \\
&\quad + \left| \int_0^t \int_U \big\langle X^n_s, \sigma(s, X^n_s, \mu^n_s, u) \big\rangle M^n(\mathrm{d}u, \mathrm{d}s) \right|^q \\
&\quad + \int_0^t \int_U \big|\sigma(s, X^n_s, \mu^n_s, u)\big|^{2q} Q^n_s(\mathrm{d}u) \, \mathrm{d}s 
\Bigg)
\end{align*}
By Burkholder-Davis-Gundy inequality and assumption \hyperlink{A.3}{(A.3)} we deduce 
\begin{equation}
\begin{split}
\mathbb{E}^{\mathbb{Q}^n}\Big( \sup_{ 0 \leq t \leq T} \Big(|X^n_t|^2+2\int_{0}^{t}\langle & X^n_s, \mathrm{d}K^n_s \rangle \Big)^q \Big) \leq C\Bigg(1+\mathbb{E}^{\mathbb{Q}^n}|X_0|^{2q}+ \mathbb{E}^{\mathbb{Q}^n}\left( \int_{0}^{T} |X^n_s|^{2q} \mathrm{d}s \right)\\
&+ \mathbb{E}^{\mathbb{Q}^n}\left( \int_0^T\int_{U}^{} \big|X^n_s\big|^2\big|\sigma(s,X^n_s, \mu^n_s, u)\big|^2Q^n_s(\mathrm{d}u)\mathrm{d}s\right)^{q/2}\Bigg)\\
&\leq C\Bigg(1+\mathbb{E}^{\mathbb{Q}^n}|X_0|^{2q}+\mathbb{E}^{\mathbb{Q}^n}\left( \int_{0}^{T} |X^n_s|^{2q} \mathrm{d}s \right)\\ &+\mathbb{E}^{\mathbb{Q}^n}\left(\sup_{0 \leq t \leq T} |X^n_t|^q\left(\int_0^T \int_{U}^{}\big|\sigma(s,X^n_s, \mu^n_s, u)\big|^2Q^n_s(\mathrm{d}u)\mathrm{d}s\right)^{q/2}\right) \Bigg)\\
&\leq C\Bigg(1+\mathbb{E}^{\mathbb{Q}^n}|X_0|^{2q}+\mathbb{E}^{\mathbb{Q}^n}\left( \int_{0}^{T} |X^n_s|^{2q} \mathrm{d}s \right) \Bigg)  +\dfrac{1}{2} \mathbb{E}^{\mathbb{Q}^n}\left(\sup_{0 \leq t \leq T} |X^n_t|^{2q}\right)
\label{eqlemma1.1}
\end{split}
\end{equation}
In the last inequality, we used the standard inequality $ 2ab \leq \dfrac{a^2}{\epsilon} + \epsilon b^2 $, which holds for all $ \epsilon > 0 $ and $ a, b \in \mathbb{R} $. We also used the growth assumptions on $ \sigma $ stated in Assumption \hyperlink{A.3}{(A.3)}.

Since $ D $ is convex, the squared distance function $ x \mapsto \operatorname{dist}^2(x, \bar{D}) $ is differentiable, and its gradient is given by
\begin{equation*}
\nabla \operatorname{dist}^2(x, \bar{D}) = 2(x - \pi_{\bar{D}}(x)),
\end{equation*}
where $ \pi_{\bar{D}}(x) $ is the projection of $ x $ onto the closed convex set $ \bar{D} $.\\ Moreover, since $ \operatorname{dist}^2(x, \bar{D}) $ is convex, the following inequality holds
\begin{equation}\label{convex}
\langle x - y, 2(x - \pi_{\bar{D}}(x)) \rangle \geq \operatorname{dist}^2(x, \bar{D}) \geq 0, \quad \forall x \in \mathbb{R}^d,\, y \in \bar{D}.
\end{equation}
From inequality~\eqref{convex} and the assumption that $ 0 \in D $, we conclude that
\begin{equation*}
\displaystyle\int_{0}^{t}\langle X^n_s, \mathrm{d}K^n_s \rangle = \displaystyle\int_{0}^{t}\langle X^n_s, n(X^n_s-\pi_{\bar{D}}(X^n_s) ) \rangle \mathrm{d}s \geq 0.
\end{equation*}
It follows from \eqref{eqlemma1.1} that
\begin{align*}
\mathbb{E}^{\mathbb{Q}^n}\left( \sup_{ 0 \leq t \leq T} |X^n_t|^{2q} \right) &\leq C \left(1+\mathbb{E}^{\mathbb{Q}^n}|X_0|^{2q}+\mathbb{E}^{\mathbb{Q}^n}\left( \int_{0}^{T} |X^n_s|^{2q} \mathrm{d}s \right) \right) \\
&\leq C \left(1+(|\lambda|^{q'})^{2q/q'}+\mathbb{E}^{\mathbb{Q}^n}\left( \int_{0}^{T} |X^n_s|^{2q} \mathrm{d}s \right) \right),
\end{align*}
and from the Gronwall lemma,
\begin{equation}\label{eqlemma1.2}
\mathbb{E}^{\mathbb{Q}^n}\left( \sup_{ 0 \leq t \leq T} |X^n_t|^{2q} \right) \leq C.
\end{equation}
Combining equations \eqref{eqlemma1.1} and \eqref{eqlemma1.2}, we obtain 
\begin{equation}\label{eqlemma1.3}
\mathbb{E}^{\mathbb{Q}^n}\left( \displaystyle\int_{0}^{T}\langle X^n_s, \mathrm{d}K^n_s \rangle\right)^q \leq C.
\end{equation}
Because $ D $ is open and contains $ 0 $, there exists $ \epsilon > 0 $ such that
\begin{equation*}
\epsilon \frac{K^n_t - K^n_s}{\lvert K^n_t - K^n_s \rvert} \in D.
\end{equation*}
Applying inequality \eqref{convex} with $ y = \epsilon \frac{K^n_t - K^n_s}{\lvert K^n_t - K^n_s \rvert} $, we obtain
\begin{equation}
\begin{split}
&\int_{s}^{t}\langle X^n_r-y, \mathrm{d}K^n_r \rangle \geq 0\\
&\int_{s}^{t}\langle X^n_r, \mathrm{d}K^n_r \rangle \geq \int_{s}^{t}\langle y, \mathrm{d}K^n_r \rangle = \epsilon |K^n_t-K^n_s|.
\end{split} \label{eqlemma1.4}
\end{equation}
Recall that,
\begin{equation}\label{eqlemma1.5}
|K^n|_{[0, T]} =\sup_{\pi}|\sum_{i=0}^{m} K^n_{t_{i+1}}-K^n_{t_{i}}|\leq \sup_{\pi}\sum_{i=0}^{m} |K^n_{t_{i+1}}-K^n_{t_{i}}| 
\end{equation}
where the supremum is taken over all subdivisions $\pi$ of $[0, T]$.\\
Therefore, combining \eqref{eqlemma1.3}, \eqref{eqlemma1.4} and \eqref{eqlemma1.5}, we drive 
\begin{equation*}
\mathbb{E}^{\mathbb{Q}^n}\left( |K^n|^q_{[0, T]} \right) \leq C.
\end{equation*}
Since, 
\begin{equation*}
K^n_t = \displaystyle \int_{0}^{t} \displaystyle\int_{U}^{} b(s,X^n_s,\mu^n_s, u)Q^n_s(\mathrm{d}u)\mathrm{d}s+ \int_{0}^{t}\int_{U}^{}\sigma(s,X^n_s,\mu^n_s, u) M^n(\mathrm{d}u,\mathrm{d}s)+X^n_0-X^n_t
\end{equation*}
It is easy to verify that 
\begin{equation*}
\mathbb{E}^{\mathbb{Q}^n}\left(\sup_{0 \leq t \leq T} |K^n_t|^{2q} \right) \leq C.
\end{equation*}
\end{proof}
\begin{lemma}\label{lemma2}
$(X^n, K^n)_{n \geq 1}$ is tight with respect to the $S$-topology and $(M^n(\mathrm{d}u, \mathrm{d}t))_{n \geq 1}$ is tight on space $C_{S'}=\mathcal{C}([0, T], S')$ of continuous functions from $[0, T]$ with values in $S'$ the topological dual of the Schwartz space $S$ of rapidly decreasing functions.
\end{lemma}
\begin{proof}
The martingale measure $(M^n(\mathrm{d}u, \mathrm{d}t))_{n \geq 0}$ can be interpreted as a process taking values in the space of distributions. Since $U$ is compact and the predictable quadratic variation satisfies $\langle M^n(\mathrm{d}u, \mathrm{d}t) \rangle = Q^n_t(\mathrm{d}u)\, \mathrm{d}t$, 
we have that, for every $p_0 \geq 1$,
\begin{equation}\label{Condition 2.1}
\sup_{n} \mathbb{E}^{\mathbb{Q}^n} \left(
\int_{0}^{T} \int_{U} \frac{1}{1+|u|^{p_0}}\, Q^n_t(\mathrm{d}u)\, \mathrm{d}t
\right)
\leq C_T \sup_{u \in U} \frac{1}{1+|u|^{p_0}} < \infty .
\end{equation}
this ensures that Condition 2.1 in \cite{Cho} is satisfied. Therefore, by \cite[Lemma 2.1]{Cho}, the family $(M^n(\mathrm{d}u, \mathrm{d}t))_{n \geq 0}$ defines a sequence of random variables taking values in $C_S'$.\\
 Therefore, by \cite[Theorem 3.1]{mitoma}, the sequence $(M^n(\mathrm{d}u, \mathrm{d}t))_{n \geq 0}$ is tight in $C_S' $ provided that, for every test function $\phi \in S$, the sequence $(M^n(\phi))_{n \geq 0}$ is tight in $C([0, T], \mathbb{R}^d)$, where $M^n(\phi)$ is defined as
\begin{equation*}
M^n_t(\phi) := \int_0^t \int_U \phi(u) M^n(\mathrm{d}u, \mathrm{d}s).
\end{equation*}
To prove the tightness of $(M^n(\phi))_{n \geq 0}$, we apply Aldous’s criterion from \cite[Theorem 16.10]{Bellin}. \\
Fix $s < t$ and $q > 1$. Using the Burkholder-Davis-Gundy inequality, we obtain
\begin{align*}
\mathbb{E}^{\mathbb{Q}^n}\left(|M^n_t(\phi) - M^n_s(\phi)|^{2q}\right) & \leq C_q \mathbb{E}^{\mathbb{Q}^n} \left( \left( \int_s^t \int_U |\phi(u)|^2 Q^n_s (\mathrm{d}u) \mathrm{d}s \right)^q \right)\\
&\leq C_q \mathbb{E}^{\mathbb{Q}^n} \left( \left( \int_s^t \sup_{u \in U} |\phi(u)|^2 \mathrm{d}t \right)^q \right)\\
&\leq C_q \sup_{u \in U} |\phi(u)|^{2q} |t - s|^q \leq K_q |t - s|^q.
\end{align*}
Thus, this implies that the sequence $(M^n(\phi))_{n \geq 0}$ is tight in $C([0, T], \mathbb{R}^d)$.

To prove that $(X^n, K^n)$ is $S$-tight, we apply Remark (A.1) in \cite{Lejay}. 
By Lemma \ref{lemma1}, the characterization given in this remark directly applies to $K^n$; hence $(K^n)_{n\ge1}$ is $S$-tight.\\
We now estimate the conditional variation of $(X^n)_{n\ge1}$. For any $n \ge 1$, since the conditional variation of a martingale is zero, we have
\begin{equation*}
CV_T(X^n) \le CV_T\left(\int_0^{\cdot}\int_U b(s, X^n_s, \mu^n_s, u)\,Q^n_s(\mathrm{d}u)\,\mathrm{d}s\right) + CV_T(K^n).
\end{equation*}
Under the linear growth condition \hyperlink{A.3}{(A.3)} on $b$, there exists a constant $C>0$ such that
\begin{equation*}
CV_T\left(\int_0^{\cdot}\int_U b(s, X^n_s, \mu^n_s, u)\,Q^n_s(\mathrm{d}u)\,\mathrm{d}s\right)
\le C\left(1+\Big(\mathbb{E}^{\mathbb{Q}^n}\sup_{0\le t\le T}|X^n_t|^2\Big)^{1/2}\right).
\end{equation*}
Hence,
\begin{equation*}
CV_T(X^n)
\le C\left(
1+\sup_n \Big(\mathbb{E}^{\mathbb{Q}^n}\sup_{0\le t\le T}|X^n_t|^2\Big)^{1/2}
+\sup_n \mathbb{E}^{\mathbb{Q}^n}|K^n|_{[0,T]}
\right).
\end{equation*}
By Lemma \ref{lemma1}, the right-hand side is finite, so the condition of Remark (A.1) in \cite{Lejay} is satisfied. 
Therefore, $(X^n)_{n\ge1}$ is $S$-tight, and consequently, $(X^n, K^n)$ is $S$-tight.
\end{proof}

\begin{lemma}\label{lemma3}
The sequence of processes $(Q^n)_{n \geq 0}$ is tight in $\mathcal{U}$.
\end{lemma}
\begin{proof}
The sequence $(Q^n)_{n \ge 0}$ takes values in the compact space $\mathcal{U}$ of measures. 
Since each $Q^n$ is a $\mathcal{U}$-valued random variable, the corresponding family of probability distributions is tight by Prokhorov’s theorem.
\end{proof}
\subsection{Admissibility and optimality of the limit of $(P^n)_{n\ge1}$}\label{subsection 4.2}
The proof of Theorem~\ref{MFGR} is carried out in two steps. 
First, we show that the limit of $(P^n)_{n \geq 1}$ is admissible. 
Second, we prove that this limit is optimal.\\

\noindent \textbf{Step 1: Admissibility} 
 
\begin{proof}
Lemmas \ref{lemma2} and  \ref{lemma3} implies that the sequence of processes $(X^n, K^n, M^n, Q^n, X^n_0)$ is tight on the space $\Gamma:= [\mathcal{D}([0,T], \mathbb{R^d})]^2 \times \mathcal{C}_{S'}([0,T]) \times \mathcal{V} \times \bar{D}$.
By Skorokhod representation Lemma, there exists a probability space $(\hat{\Omega},\hat{\mathcal{F}}, \hat{\mathbb{P}})$, a sequence $(\hat{X}^n, \hat{K}^n, \hat{M}^n, \hat{Q}^n, \hat{X}^n_0)$ and $(\hat{X}, \hat{K}, \hat{M}, \hat{Q}, \hat{Y})$ defined on this space such that 
\begin{description}
\item[E.1)] for each $n \in \mathbb{N},$
\begin{equation}
Law(\hat{X}^n, \hat{K}^n, \hat{M}^n, \hat{Q}^n, \hat{X}^n_0) = Law(X^n, K^n, M^n, Q^n, X^n_0).
\end{equation}
\item[E.2)] \label{eq:E.2}
there exists a sub-sequence $(\hat{X}^{n_k}, \hat{K}^{n_k}, \hat{M}^{n_k}, \hat{Q}^{n_k}, \hat{X}^{n_k}_0)$ of $(\hat{X}^n, \hat{K}^n, \hat{M}^n, \hat{Q}^n, \hat{X}^n_0)$, denoted again by $(\hat{X}^n, \hat{K}^n, \hat{M}^n, \hat{Q}^n, \hat{X}^n_0)$, it converges almost surely $\hat{\mathbb{P}}-a.s.$ to $(\hat{X}, \hat{K}, \hat{M}, \hat{Q}, \hat{Y})$ in the space $\Gamma$.
\item[E.3)] $(\hat{X}^n, \hat{K}^n)$ converges to $(\hat{X}, \hat{K})$, $\mathrm{d}t\times \hat{\mathbb{P}}-a.s.$, and $(\hat{X}^n_T, \hat{K}^n_T) \to (\hat{X}_T, \hat{K}_T)$ $\hat{\mathbb{P}}-a.s.$
\end{description}


Let $(X', K')$ denote the unique solution to the reflected SDE
\begin{equation}\label{optimaleq2}
X'_t + K'_t= \hat{Y} 
+ \displaystyle\int_{0}^{t}\!\int_{U} b(s, \hat{X}_s, \mu_s, u)\,\hat{Q}_s(\mathrm{d}u)\,\mathrm{d}s 
+ \displaystyle\int_{0}^{t}\!\int_{U} \sigma(s, \hat{X}_{s}, \mu_{s}, u)\,\hat{M}(\mathrm{d}u,\mathrm{d}s) 
,
\end{equation}
where $\mu_t := \hat{\mathbb{P}}_{\hat{X}_t}$.\\
The initial condition satisfies
\begin{equation*}
\hat{\mathbb{P}} \circ (\hat{Y})^{-1} 
= \lim_{n \to \infty} \hat{\mathbb{P}} \circ (\hat{X}^n_0)^{-1} 
= \lim_{n \to \infty} \hat{\mathbb{Q}}^n \circ (X^n_0)^{-1} 
= \lambda.
\end{equation*}
Equation~\eqref{equation1}, together with property \textbf{(E.1)}, implies that $(\hat{X}^n, \hat{K}^n)$ satisfy the following equation
\begin{equation*}
\hat{X}^n_t = \hat{H}^n_t+ \hat{Z}^n_t-\hat{K}^n_t, \forall t \in [0, T],
\end{equation*}
where 
\begin{align*}
\hat{Z}^n_t := &\displaystyle\int_{0}^{t}\displaystyle\int_{U}^{} b(s,\hat{X}^n_s, \mu^n_s, u)\hat{Q}^n_s(\mathrm{d}u)\mathrm{d}s+\displaystyle\int_{0}^{t}\int_{U}^{}\sigma(s,\hat{X}^n_s, \mu^n_s, u) \hat{M}^n(\mathrm{d}u,\mathrm{d}s), \,\, \hat{H}^n_t :=\hat{X}^n_0, \\
&\text{and}\quad \mu^n= \mathbb{Q}^n\circ(X^n)^{-1} = \hat{\mathbb{P}}^n\circ(\hat{X}^n)^{-1}.
\end{align*}
The following convergences hold in probability in $\mathcal{D}([0, T]; \mathbb{R}^d)$:
\begin{equation}
\begin{split}
&\displaystyle\int_{0}^{\cdot}\displaystyle\int_{U}^{} b(s,\hat{X}^n_s, \mu^n_s, u)\hat{Q}^n_s(\mathrm{d}u)\mathrm{d}s \to \displaystyle\int_{0}^{\cdot}\displaystyle\int_{U}^{} b(s,\hat{X}_s, \mu_s, u)\hat{Q}_s(\mathrm{d}u)\mathrm{d}s, \\
&\displaystyle\int_{0}^{\cdot}\int_{U}^{}\sigma(s, \hat{X}^n_s, \mu^n_s, u) \hat{M}^n(\mathrm{d}u, \mathrm{d}s) \to \displaystyle\int_{0}^{\cdot}\int_{U}^{}\sigma(s,\hat{X}_{s}, \mu_{s}, u) \hat{M}(\mathrm{d}u,\mathrm{d}s). \label{con ofz}
\end{split}
\end{equation}
The first convergence is easy to prove, then we will prove only the second.\\
We have, 
\begin{equation}
\begin{split}
&\displaystyle\int_{0}^{\cdot}\int_{U}^{}\sigma(s, \hat{X}^n_s, \mu^n_s, u) \hat{M}^n(\mathrm{d}u,\mathrm{d}s) - \displaystyle\int_{0}^{\cdot}\int_{U}^{}\sigma(s, \hat{X}_{s}, \mu_{s}, u) \hat{M}(\mathrm{d}u,\mathrm{d}s)\\
&=\displaystyle\int_{0}^{\cdot}\int_{U}^{}(\sigma(s, \hat{X}^n_s, \mu^n_s, u) -\sigma(s, \hat{X}_{s}, \mu_{s}, u))\hat{M}^n(\mathrm{d}u,\mathrm{d}s) \\
&+\displaystyle\int_{0}^{\cdot}\int_{U}^{}\sigma(s, \hat{X}_{s}, \mu_{s}, u) \hat{M}^n(\mathrm{d}u,\mathrm{d}s)- \displaystyle\int_{0}^{\cdot}\int_{U}^{}\sigma(s, \hat{X}_{s}, \mu_{s}, u) \hat{M}(\mathrm{d}u, \mathrm{d}s). \label{condition16}
\end{split}
\end{equation}
By applying the Burkholder--Davis--Gundy inequality, together with properties \textbf{(E.1)--(E.3)} and the Lipschitz condition of $\sigma$, we obtain
\begin{align*}
\mathbb{E}^{\hat{\mathbb{P}}} \Bigg(
\sup_{0 \leq t \leq T} 
\Big| \int_{0}^{t} \int_{U} 
\big( \sigma(s, \hat{X}^n_{s}, \mu^n_{s}, u) 
- \sigma(s, \hat{X}_{s}, \mu_{s}, u) \big) 
\hat{M}^n(\mathrm{d}u, \mathrm{d}s) 
\Big|^2 \Bigg)
&\le 
C\, \mathbb{E}^{\hat{\mathbb{P}}} 
\left( \int_{0}^{T} 
\big| \hat{X}^n_t - \hat{X}_t \big|^2 
\, \mathrm{d}t \right)
\\
&\longrightarrow 0,
\end{align*}
which implies that the first term in~\eqref{condition16} converges to $0$ in probability in $\mathcal{C}([0, T]; \mathbb{R}^d)$.\\
By jointly considering Condition~\eqref{Condition 2.1} and Theorem~2.1 from~\cite{Cho}, we derive
\begin{equation*}\label{secondterm}
\int_{0}^{\cdot}\int_{U} 
\sigma(s, \hat{X}_{s}, \mu_{s}, u)\, 
\hat{M}^n(\mathrm{d}u, \mathrm{d}s)
\xrightarrow{\;\mathcal{P}\;}
\int_{0}^{\cdot}\int_{U} 
\sigma(s, \hat{X}_{s}, \mu_{s}, u)\, 
\hat{M}(\mathrm{d}u, \mathrm{d}s) 
\quad \text{in probability in } \mathcal{D}([0, T]; \mathbb{R}^d).
\end{equation*}
Consequently,
\begin{equation*}
(\hat{H}^n, \hat{Z}^n) \;\xrightarrow{\mathcal{P}}\; (\hat{H}, \hat{Z})
\quad \text{in probability in } \mathcal{D}([0, T]; \mathbb{R}^{2d}),
\end{equation*}
where
\begin{align*}
\hat{Z}_t 
&:= 
\int_{0}^{t}\int_{U} 
b(s, \hat{X}_{s}, \mu_s, u)\, 
\hat{Q}_s(\mathrm{d}u)\,\mathrm{d}s
+ 
\int_{0}^{t}\int_{U} 
\sigma(s, \hat{X}_{s}, \mu_{s}, u)\, 
\hat{M}(\mathrm{d}u, \mathrm{d}s),
\\[4pt]
\hat{H}_t &:= \hat{Y}.
\end{align*}

On the other hand, for every $q \ge 0$ and for each discrete predictable process 
of the form
\begin{equation*}
\hat{U}^{n}_{s}
=
\hat{U}^{n}_{0}
+
\sum_{i=0}^{k}
\hat{U}^{n}_{i}\,
\mathbf{1}_{\{t_i < s \le t_{i+1}\}},
\end{equation*}
where 
\begin{equation*}
0 = t_0 < t_1 < \cdots < t_k = q, 
\qquad
\hat{U}^{n}_{i} \text{ is } \hat{\mathcal{F}}^{n}_{t_i}\text{-measurable},
\qquad
|\hat{U}^{n}_{i}| \le 1,
\quad i \in \{0,\dots,k\},\ n,k \in \mathbb{N},
\end{equation*}
we have
\begin{align*}
\mathbb{E}^{\hat{\mathbb{P}}} 
\Big| \int_{0}^{q} \hat{U}^n_{s} \, \mathrm{d} \hat{Z}^n_s \Big|^2 
&\leq 
2\,\mathbb{E}^{\hat{\mathbb{P}}} 
\Bigg|
\int_{0}^{q}\int_{U} 
\hat{U}^n_s\, 
b(s, \hat{X}^n_s, \mu^n_s, u)\, 
\hat{Q}^n_s(\mathrm{d}u)\,\mathrm{d}s 
\Bigg|^2
\\
&\quad +\, 
2\,\mathbb{E}^{\hat{\mathbb{P}}}
\Bigg|
\int_{0}^{q}\int_{U} 
\hat{U}^n_s\, 
\sigma(s, \hat{X}^n_s, \mu^n_s, u)\, 
\hat{M}^n(\mathrm{d}u, \mathrm{d}s)
\Bigg|^2
\\[4pt]
&\le 
C\left( 1 + 
\sup_{n} 
\mathbb{E}^{\hat{\mathbb{P}}}
\sup_{0 \le s \le q} 
|\hat{X}^n_s|^2 
\right)
< \infty.
\end{align*}
The last inequality follows from assumption~\hyperlink{A.3}{(A.3)} and from~\eqref{eq3.1}. 
Therefore, the sequence $(\hat{Z}^n)$ of 
$(\hat{\mathcal{F}}^{n}_{t})$-adapted semimartingales 
satisfies Stricker’s condition~\textbf{(UT)}~\cite{stricker}.\\
We now verify that the assumptions of~\cite[Theorem~4.3]{Lauka} are fulfilled. 
Hence, we can apply points~(ii) and~(iii) of that theorem. 
Since the reflected SDE~\eqref{optimaleq2} admits a unique weak solution, 
point~(iii) implies that
\begin{equation*}
\hat{X}^n \to X'
\quad \text{in distribution with respect to the S-topology.}
\end{equation*}
Applying point (ii) with $ g := I_{\text{id}}$, we obtain
\begin{equation*}
(\hat{X}^n_{t_1}, \hat{X}^n_{t_2}, ..., \hat{X}^n_{t_q}, \hat{Z}^n) \to  (X'_{t_1}, X'_{t_2}, ..., X'_{t_q}, \hat{Z}) \quad \text{in distribution in} \quad \mathbb{R}^{dq} \times \mathcal{D}([0, T], \mathbb{R}^d),
\end{equation*}
we may choose the time points $ t_0, t_1, ..., t_q$ to form a partition of the interval $[0, T]$, with $ t_0 = 0 < t_1 < \cdots < t_q = T$, since both $ \hat{Z} $ and $\hat{H}$ are continuous.\\
Since the sequence $ (\hat{X}^n, \hat{Z}^n) $ is tight with respect to the S-Skorokhod topology, it follows from \cite[Theorem 3.5]{jakub} that
\begin{equation*}
(\hat{X}^n, \hat{Z}^n) \xrightarrow{d} (\hat{X}, \hat{Z}) \quad \text{in distribution with respect to the S-topology}.
\end{equation*}
Using the fact that addition is sequentially continuous in the S-topology, according to Remark~3.12 in~\cite{jakub}, we get
\begin{equation*}
\hat{K}^n = \hat{Z}^n - \hat{X}^n \to \hat{Z} - X' = K' \quad \text{in distribution with respect to the S-topology}.
\end{equation*}
By applying the same arguments as those used in the proof of Lemma~3.5 in~\cite{bahlali-PDE}, we deduce 
\begin{equation*}\label{adjointopt}
(\hat{X}^n, \hat{K}^n) \xrightarrow{d} (X', K')
\quad \text{in distribution with respect to the uniform topology}.
\end{equation*}
From the uniqueness of the limit in law, property~\textbf{(E.2)} implies that
\begin{equation}\label{same-law-in (X,K)}
\mathrm{Law}(\hat{X}, \hat{K}) = \mathrm{Law}(X', K').
\end{equation}
Combining this identity with the uniform estimates obtained in~\eqref{eq3.1}, we deduce that
\begin{equation*}
\mathbb{E}^{\hat{\mathbb{P}}}
\Bigg(
\sup_{0 \le s \le T} 
|\hat{X}^n_s - \hat{X}_s|^2
+
\sup_{0 \le s \le T} 
|\hat{K}^n_s - \hat{K}_s|^2
\Bigg)
\longrightarrow 0.
\end{equation*}
Consequently,
\begin{equation*}
\mu^n \longrightarrow \mu 
\quad \text{in } \mathcal{P}_2\big(\mathcal{C}^d\big).
\end{equation*}
Hence, property~1) in Theorem~\ref{MFGR} holds.\\
Furthermore, according to~\eqref{optimaleq2} and \eqref{same-law-in (X,K)}, 
the pair $(\hat{X}, \hat{K})$ satisfies the following reflected SDE:
\begin{equation}\label{reflectedSDE}
\begin{cases}
\mathrm{d}\hat{X}_t+ 
\mathrm{d}\hat{K}_t
= 
\displaystyle\int_{U} b(t, \hat{X}_t, \mu_t, u)\, 
\hat{Q}_t(\mathrm{d}u)\,\mathrm{d}t
+ 
\displaystyle\int_{U} \sigma(t, \hat{X}_t, \mu_t, u)\,
\hat{M}(\mathrm{d}u,\mathrm{d}t)
, \\[6pt]
\hat{\mathbb{P}} \circ (\hat{X}_0)^{-1} 
= \hat{\mathbb{P}} \circ (\hat{Y})^{-1} = \lambda, \\[4pt]
\end{cases}
\end{equation}
It follows from Itô’s formula that
\begin{equation*}
P:=\hat{\mathbb{P}} \circ (\hat{X}, \hat{K}, \hat{Q})^{-1}
\in 
\mathcal{R}_{\mathrm{ref}(D)}(\mu), \quad \text{where} \quad \mu = \hat{\mathbb{P}} \circ (\hat{X})^{-1}= P\circ(X)^{-1}.
\end{equation*}
By the continuity of the mappings  
\begin{equation*}
(x, \mu, q) \longmapsto 
\int_0^{T} \int_U 
q_t(\mathrm{d}u)\, 
f(t, x_t, \mu_t, u)\,\mathrm{d}t 
+ g(x_T, \mu_T),
\qquad
(x, k, \mu) \longmapsto 
\int_{0}^{T} 
h(t, x_t, \mu_t)\,\mathrm{d}k_t,
\end{equation*}
(the continuity of the first mapping follows from the same arguments as in Lemma A.4 of \cite{luciano}, and that of the second from Lemma 4.6 in \cite{badr}), together with the assumptions on $(f, g, h)$ and the uniform estimates in~\eqref{eq3.1}, we deduce that the desired convergence holds,
\begin{align}\label{Con-cost1}
\lim_{n \to \infty} J^n(\mu^n, P^n)
&= 
\lim_{n \to \infty} 
\mathbb{E}^{\mathbb{Q}^n}
\Bigg(
g(X^n_T, \mu^n_T)
+
\int_0^T \int_U 
Q^n_t(\mathrm{d}u)\,
f(t, X^n_t, \mu^n_t, u)\,\mathrm{d}t
+
\int_0^T 
h(t, X^n_t, \mu^n_t)\,
\mathrm{d}K^n_t
\Bigg)
\notag\\[4pt]
&=
\lim_{n \to \infty} 
\mathbb{E}^{\hat{\mathbb{P}}}
\Bigg(
g(\hat{X}^n_T, \mu^n_T)
+
\int_0^T \int_U 
\hat{Q}^n_t(\mathrm{d}u)\,
f(t, \hat{X}^n_t, \mu^n_t, u)\,\mathrm{d}t
+
\int_0^T 
h(t, \hat{X}^n_t, \mu^n_t)\,
\mathrm{d}\hat{K}^n_t
\Bigg)
\notag\\[4pt]
&=\mathbb{E}^{\hat{\mathbb{P}}}
\Bigg(
g(\hat{X}_T, \mu_T)
+
\int_0^T \int_U 
\hat{Q}_t(\mathrm{d}u)\,
f(t, \hat{X}_t, \mu_t, u)\,\mathrm{d}t
+
\int_0^T 
h(t, \hat{X}_t, \mu_t)\,
\mathrm{d}\hat{K}_t
\Bigg)
\notag\\[4pt]
&=
J\left(\mu, P\right),
\end{align}
which establishes point~2) of Theorem~\ref{MFGR}.\\
Finally, it remains to verify that
\begin{equation*}
P \in \mathcal{R}^*_{\mathrm{ref}(D)}(\mu).
\end{equation*}
\end{proof}

\noindent \textbf{Step 2: Optimality}

To establish the optimality of 
$ P$,
we first introduce the following lemma, which provides a link between the cost functional associated with the penalized dynamics and that of the reflected dynamics. 
This connection enables the transfer of optimality from the penalized setting to the reflected one.
\begin{lemma}\label{adjointlemma}
For every $ P_1 \in \mathcal{R}_{\mathrm{ref}(D)}(\mu)$, 
there exists a sequence 
$ P_1^{n} \in \mathcal{R}_n(\mu^n) $ such that
\begin{equation}\label{optimalestimation}
J(\mu, P_1)
= 
\lim_{n \to \infty} J^{n}(\mu^{n}, P_1^{n}).
\end{equation}
\end{lemma}

\begin{proof}
By Proposition~\ref{rep martingale for Reflected dynamics}, 
we can construct a filtered probability space 
$(\Omega', \mathcal{F}'_t, \mathbb{Q}')$
supporting a $d$-dimensional 
$\mathcal{F}'_t$-adapted process 
$(X', K')$ and 
$m$ orthogonal $\mathcal{F}'_t$-martingale measures 
$ M' = (M'_1, \ldots, M'_m) $
on $ U \times [0,T]$ with intensity 
$ Q'_t(\mathrm{d}u)\,\mathrm{d}t$,
such that 
$ \mathbb{Q}' \circ (X', K', Q')^{-1} = P_1 $
and the following reflected SDE holds:
\begin{equation}\label{eq:refSDEX'}
\begin{cases}
\mathrm{d}X'_t+
\mathrm{d}K'_t
= 
\displaystyle\int_U b(t, X'_t, \mu_t, u)\,
Q'_t(\mathrm{d}u)\,\mathrm{d}t
+
\displaystyle\int_U \sigma(t, X'_t, \mu_t, u)\,
M'(\mathrm{d}u, \mathrm{d}t)
, \\[6pt]
\mathbb{Q}'\circ(X'_0)^{-1} = \lambda,
\end{cases}
\end{equation}
For each $ n \in \mathbb{N} $, consider the penalized SDE
\begin{equation}\label{eq:penalXn}
\begin{cases}
\mathrm{d}X'^n_t
=
\displaystyle\int_U 
\Big(b(t, X'^n_t, \mu^n_t, u)\,Q'_t(\mathrm{d}u)
- n\big(X'^n_t - \pi_{\bar{D}}(X'^n_t)\big)\Big)
\mathrm{d}t
+
\displaystyle\int_U 
\sigma(t, X'^n_t, \mu^n_t, u)\,
M'(\mathrm{d}u, \mathrm{d}t),\\
X'^n_0 = X'_{0},
\end{cases}
\end{equation}
and define 
$ P^n_1 := \mathbb{Q}' \circ (X'^n, Q')^{-1} $.
It is immediate that 
$P^n_1 \in \mathcal{R}_n(\mu^n) $.\\
Using the convergence $ \mu^n \to \mu $ and arguments analogous to those in \textbf{Step~1}, 
we deduce that
\begin{equation}\label{eq:convXnKn}
(X'^n, K'^n)
:=
\Biggl(X'^n, 
\int_0^{\cdot} n \bigl(X'^n_s - \pi_{\bar{D}}(X'^n_s)\bigr)\, \mathrm{d}s \Biggr)
\xrightarrow{d} 
(X', K'),
\end{equation}
in distribution with respect to the uniform topology.\\
Moreover, the uniform bound~\eqref{eq3.1} also holds for the sequence $(X'^n, K'^n)$.\\
By the Skorokhod representation theorem, 
there exists a probability space 
$(\tilde{\Omega}, \tilde{\mathcal{F}}, \tilde{\mathbb{P}})$ 
and random variables 
$(\tilde{X}'^n, \tilde{K}'^n)$, 
$(\tilde{X}', \tilde{K}')$
such that
\begin{equation*}
(\tilde{X}'^n, \tilde{K}'^n)
\to
(\tilde{X}', \tilde{K}')
\quad 
\text{a.s. in } 
\mathcal{C}([0,T]; \mathbb{R}^{2d}),
\end{equation*}
and
\begin{equation*}
Law(\tilde{X}'^n, \tilde{K}'^n)
= 
Law(X'^n, K'^n),
\ 
Law(\tilde{X}', \tilde{K}')
= 
Law(X', K').
\end{equation*}
Since $ f, g, $ and $ h $ are continuous and satisfy the growth conditions in assumption~\hyperlink{A.3}{(A.3)}, 
it follows that
\begin{align*}
& g(\tilde{X}'^n_T, \mu^n_T) - g(\tilde{X}'_T, \mu_T)
+ 
\int_0^T\int_U 
Q'_t(\mathrm{d}u)
\big(f(t, \tilde{X}'^n_t, \mu^n_t, u)
    - f(t, \tilde{X}'_t, \mu_t, u)\big)\,\mathrm{d}t
\\
&\quad 
+ 
\int_0^T 
h(t, \tilde{X}'^n_t, \mu^n_t)\,\mathrm{d}\tilde{K}'^n_t
- 
\int_0^T 
h(t, \tilde{X}'_t, \mu_t)\,\mathrm{d}\tilde{K}'_t
\;\longrightarrow\; 0,
\end{align*}
and the corresponding terms are uniformly integrable. 
Hence, by the dominated convergence theorem,
\begin{align*}
\lim_{n \to \infty} J^n(\mu^n, P^n_1)
&= 
\lim_{n \to \infty}
\mathbb{E}^{\mathbb{Q}'}
\Bigg(
g(X'^n_T, \mu^n_T)
+ 
\int_0^T\int_U 
Q'_t(\mathrm{d}u)\,
f(t, X'^n_t, \mu^n_t, u)\,\mathrm{d}t
+ 
\int_0^T 
h(t, X'^n_t, \mu^n_t)\,
\mathrm{d}K'^n_t
\Bigg)
\\[3pt]
&= 
\mathbb{E}^{\tilde{\mathbb{P}}}
\Bigg(
g(\tilde{X}'_T, \mu_T)
+ 
\int_0^T \int_U 
Q'_t(\mathrm{d}u)\,
f(t, \tilde{X}'_t, \mu_t, u)\,\mathrm{d}t
+ 
\int_0^T 
h(t, \tilde{X}'_t, \mu_t)\,
\mathrm{d}\tilde{K}'_t
\Bigg)
\\[3pt]
&= 
J(\mu, P_1).
\end{align*}
This establishes the desired result~\eqref{optimalestimation}.
\end{proof}

\begin{proof}[Proof of Step~2]
Let $ P_1 \in \mathcal{R}_{\mathrm{ref}(D)}(\mu) $ be arbitrary. 
By Lemma~\ref{adjointlemma}, there exists a sequence 
$ P_1^n \in \mathcal{R}_n(\mu^n) $ such that
\begin{equation}\label{optimalofP}
J(\mu, P_1) = \lim_{n \to \infty} J^n(\mu^n, P_1^n).
\end{equation}
Since $ P^n $ is optimal for each $n$, it follows that
\begin{equation*}
J^n(\mu^n, P^n) \le J^n(\mu^n, P_1^n),
\qquad n \in \mathbb{N}.
\end{equation*}
Using the convergence relations~\eqref{Con-cost1} and~\eqref{optimalofP}, we obtain
\begin{equation*}
J(\mu, P)
= 
\liminf_{n \to \infty} J^n(\mu^n, P^n)
\le 
\lim_{n \to \infty} J^n(\mu^n, P_1^n)
= 
J(\mu, P_1).
\end{equation*}
Therefore, $P$ belongs to $\mathcal{R}^{*}_{\mathrm{ref}(D)}(\mu)$.
\end{proof}

\section{Existence of Markovian and Strict Markovian MFG Solutions}
\label{sec-mark}

In this section, we establish the existence of a strict Markovian MFG solution with reflected dynamics. Throughout this section, we maintain Assumption~\ref{ass:A}, which guarantees the existence of a relaxed MFG solution by Theorem~\ref{MFGR}.

To construct a Markovian relaxed MFG solution, we employ a penalization approach. Specifically, we first introduce a penalized reflected SDE , and then apply a mimicking theorem to the penalized dynamics in order to obtain a relaxed Markovian MFG formulation. Subsequently, by letting the penalization parameter tend to infinity, we recover a relaxed Markovian MFG corresponding to the original reflected system. Finally, a measurable selection argument allows us to derive the existence of a strict-Markovian MFG solution.

Applying the mimicking theorem to the penalized system yields a model whose coefficients are only measurable, due to the regularity of the mapping $(t,x) \mapsto q(t,x)$. This limited smoothness necessitates additional analytical arguments. To address these difficulties, we draw on methods developed for penalization schemes in reflected SDEs with merely measurable coefficients (see, e.g., \cite{mouchtabih,Slominski-measurable}). In order to ensure the convergence of the penalization procedure, we introduce the following structural assumption.

\begin{assumption}{\textbf{(M)}}
\label{ass:M}
The diffusion coefficient $\sigma$ is uniformly elliptic: there exists a constant $\beta > 0$ such that
\begin{equation*}
\sigma(t, x, \mu, u)\sigma(t, x, \mu, u)^{\top} \geq \beta I_d, 
\quad 
\forall (t, x, \mu, u) \in [0,T]\times\mathbb{R}^d\times\mathcal{P}_2(\mathbb{R}^d)\times U.
\end{equation*}
\end{assumption}
 
\begin{assumption}{\textbf{(C)}}
\label{ass:C}
For all $(t, x, \mu) \in [0,T] \times \mathbb{R}^d \times \mathcal{P}_2(\mathbb{R}^d)$, the subset
\begin{equation*}
\mathcal{S}(t,x,\mu)
:= 
\bigl\{
\bigl(b(t,x,\mu,\alpha),\ \sigma\sigma^{\top}(t,x,\mu,\alpha),\ z\bigr)
:\ \alpha \in U,\ z \ge f(t,x,\mu,\alpha)
\bigr\}
\subset \mathbb{R}^d \times \mathbb{R}^{d \times d} \times \mathbb{R}
\end{equation*}
is convex.
\end{assumption}

\begin{theorem}\label{theorem:Markovian}
Let Assumptions~\ref{ass:A} and~\ref{ass:M} hold. Then there exists a relaxed Markovian MFG solution. 
If, in addition, Assumption~\ref{ass:C} holds, there exists a \emph{strict} Markovian MFG solution.
\end{theorem}

The proof of Theorem~\ref{theorem:Markovian} relies on the following auxiliary result.

\begin{lemma}\label{lemma:adj-Markovian}
Assume that Assumptions~\ref{ass:A} and~\ref{ass:M} are satisfied. 
Let $\mu \in \mathcal{P}_2(\mathcal{C}^d)$ and $P \in \mathcal{R}_{\mathrm{ref}(D)}(\mu)$. 
Then there exist a measurable mapping 
\begin{equation*}
q : [0, T] \times \mathbb{R}^d \longrightarrow \mathcal{P}(U)
\end{equation*}
and a probability measure $\hat{P} \in \mathcal{R}_{\mathrm{ref}(D)}(\mu)$ such that:
\begin{enumerate}[label=(\roman*)]
    \item $\hat{P}\big(Q(\mathrm{d}t, \mathrm{d}u) = \mathrm{d}t\, q(t,X_t)(\mathrm{d}u)\big) = 1$;
    \item $J(\mu, \hat{P}) = J(\mu, P)$;
    \item $\hat{P} \circ X_t^{-1} = P \circ X_t^{-1}$ for every $t \in [0, T]$.
\end{enumerate}
\end{lemma}
\begin{proof} The proof follows by adapting the arguments of Corollary~3.8 in Lacker~\cite{laker} to our setting. By Theorem~2.5 of Karoui~\cite{Mel}, there exists a measurable function 
\begin{equation*}
\bar{\sigma} : [0,T] \times \mathbb{R}^d \times \mathcal{P}(\mathbb{R}^d) \times \mathcal{P}(U) \to \mathbb{R}^{d \times m}, 
\end{equation*}
which is continuous in $(x,\mu,q)$ for each $t$, and satisfies 
\begin{equation*} 
\bar{\sigma} \bar{\sigma}^{\top}(t, x, \mu, q) = \int_{U} q(\mathrm{d}u)\, \sigma\sigma^{\top}(t, x, \mu, u), \quad\text{and}\quad \bar{\sigma} \bar{\sigma}^{\top}(t, x, \mu, \delta_{u}) = \sigma\sigma^{\top}(t, x, \mu, u), 
\end{equation*} 
for every $(t, x, \mu, u) \in [0,T] \times \mathbb{R}^d \times \mathcal{P}(\mathbb{R}^d) \times U$.\\
Moreover, we may construct a filtered probability space $(\Omega, \mathcal{F}_t, \mathbb{Q})$ supporting a $d$-dimensional $\mathcal{F}_t$-adapted process $(X, K)$ and a Wiener process $W$, such that
\begin{equation*} 
\mathbb{Q} \circ (X, K, Q)^{-1} = P 
\end{equation*}
and $(X,K)$ satisfies the reflected SDE \begin{equation} \label{Penalized-No-mimicking}  \displaystyle \mathrm{d}X_t+ \mathrm{d}K_t = \int_{U} b(t, X_t, \mu_t, u)\, Q_t(\mathrm{d}u)\, \mathrm{d}t + \bar{\sigma}(t, X_t, \mu_t, Q_t)\, \mathrm{d}W_t .
\end{equation}
For $n \in \mathbb{N}$, consider the penalized SDE 
\begin{equation} \label{Penalized-SDE} \mathrm{d}X^n_t = \int_{U} \big( b(t, X^n_t, \mu_t, u)\, Q_t(\mathrm{d}u) - n \big( X^n_t - \pi_{\bar{D}}(X^n_t) \big) \big) \mathrm{d}t + \bar{\sigma}(t, X^n_t, \mu_t, Q_t) \, \mathrm{d}W_t. 
\end{equation} 
 Since $x \mapsto x - \pi_{\bar{D}}(x)$ is Lipschitz, \eqref{Penalized-SDE} admits a unique strong solution $X^n$. Define \begin{equation*}
K^n_s := \int_{0}^{s} n \big( X^n_t - \pi_{\bar{D}}(X^n_t) \big) \, \mathrm{d}t.
\end{equation*}
It follows that $(X^n, K^n, Q)$ satisfies a penalized version of \eqref{Penalized-No-mimicking}, and the following convergence holds (by arguments similar to those in Subsection~\ref{subsection 4.2}) \begin{equation}\label{conv1}
P^n := \mathbb{Q} \circ (X^n, K^n, Q)^{-1} \; \longrightarrow \; P \quad \text{in } \mathcal{P}_{2}(\Omega). \end{equation} Analogously to Corollary~3.8 in ~\cite{laker}, there exists a measurable function 
\begin{equation*} 
q : [0, T] \times \mathbb{R}^d \longrightarrow \mathcal{P}(U) 
\end{equation*}
such that 
\begin{equation*}
q(t, X^n_t) = \mathbb{E}^{\mathbb{Q}}\left(\, Q_t \,\middle|\, X^n_t \right), \quad \mathrm{d}t \times \mathrm{d}\mathbb{Q}\text{-a.s.} 
\end{equation*} 
Hence, $\mathrm{d}t \times \mathrm{d}\mathbb{Q}\text{-a.s.},$ we have
\begin{equation*} 
\int_{U} b\big(t, X^n_{t}, \mu_t, u\big) \, q(t, X^n_{t})(\mathrm{d}u) -n \big( X^n_t - \pi_{\bar{D}}(X^n_t) \big)= \mathbb{E}^{\mathbb{Q}}\left( \int_{U} b\big(t, X^n_{t}, \mu_{t}, u\big) \, Q_t(\mathrm{d}u) -n \big( X^n_t - \pi_{\bar{D}}(X^n_t) \big)\ \middle|\ X^n_{t} \right),
\end{equation*} 
and 
\begin{align*} \bar{\sigma} \bar{\sigma}^{\top}(t, X^n_t, \mu_t, q(t, X^n_t))&= \int_{U} q(t, X^n_t)(\mathrm{d}u)\, \sigma\sigma^{\top}(t, X^n_t, \mu_t, u)\\ &= \mathbb{E}^{\mathbb{Q}}\left( \int_{U} Q_t(\mathrm{d}u) \sigma\sigma^{\top}(t, X^n_t, \mu_t, u) \,\ \middle|\ X^n_{t} \right)\\ &=\mathbb{E}^{\mathbb{Q}}\left(\bar{\sigma}\bar{\sigma}^{\top}(t, X^n_t, \mu_t, Q_t) \,\ \middle|\ X^n_{t} \right).
\end{align*} 
By the mimicking theorem, there exists another filtered probability space $(\hat{\Omega}^n, \hat{\mathcal{F}}^n_t, \hat{\mathbb{Q}}^n)$ supporting a $\hat{\mathcal{F}}^n_t$-Wiener process $\hat{W^n}$ and an $\mathbb{R}^d$-valued $\hat{\mathcal{F}}^n_t$-adapted process $\hat{X}^n$ such that 
\begin{equation} \label{penalized-mimicking} \mathrm{d}\hat{X}^n_t = \int_{U} b\big(t, \hat{X}^n_t, \mu_t, u\big)\, q(t, \hat{X}^n_t)(\mathrm{d}u)\, \mathrm{d}t -n \big( \hat{X}^n_t - \pi_{\bar{D}}(\hat{X}^n_t) \big)\mathrm{d}t + \bar{\sigma}\big(t, \hat{X}^n_t, \mu_t, q(t, \hat{X}^n_t)\big)\, \mathrm{d}\hat{W}^n_t, 
\end{equation} 
and 
\begin{equation} \label{samelaw1} \hat{\mathbb{Q}}^n \circ (\hat{X}^n_t)^{-1} = \mathbb{Q} \circ (X^n_t)^{-1}, \quad \forall\, t \in [0, T].
\end{equation} 
The equation \eqref{penalized-mimicking}, can rewrite as 
\begin{equation} 
\mathrm{d}\hat{X}^n_t+ \mathrm{d}\hat{K}^n_t = b^1\big(t, \hat{X}^n_t\big) \mathrm{d}t + \bar{\sigma}^1\big(t, \hat{X}^n_t\big)\, \mathrm{d}\hat{W}^n_t ,
\end{equation}
where 
\begin{equation*} b^1\big(t, x\big):=\int_{U} b\big(t, x, \mu_t, u\big)\, q(t, x)(\mathrm{d}u), \quad\bar{\sigma}^1\big(t, x\big):=\bar{\sigma}\big(t, x, \mu_t, q(t, x)\big) 
\end{equation*}
and 
$$\hat{K}^n_s:= \displaystyle\int_{0}^{s} n \big( \hat{X}^n_t - \pi_{\bar{D}}(\hat{X}^n_t) \big)\mathrm{d}t.$$
Similarly to Lemma~\ref{lemma1}, one can show that for every $q \ge 1$ such that $2q \in [2, q']$, there exists a constant $C > 0$ such that 
\begin{equation}\label{Dom-measurable} 
\sup_{n} \mathbb{E}^{\mathbb{Q}^n} \left( \sup_{0 \le t \le T} |\hat{X}^n_t|^{2q} + \sup_{0 \le t \le T} |\hat{K}^n_t|^{2q} + |\hat{K}^n|_{[0,T]}^{q} \right) \le C. 
\end{equation}
Note that the coefficients $(b^{1}, \bar{\sigma}^{1})$ are merely measurable, possibly discontinuous, and satisfy condition 1.3 in \cite{Slominski-measurable}. By Theorem 2.1 of \cite{Slominski-measurable}, the sequence $(\hat{X}^{n}, \hat{K}^{n})$ is therefore tight in $\mathcal{C}([0,T], \mathbb{R}^{d})$. The coefficient $\bar{\sigma}^{1}$ also satisfies the uniform ellipticity condition by assumption \ref{ass:M}. Using the arguments from the proof of Theorem 3.1 in \cite{Slominski-measurable}, we obtain 
\begin{equation}\label{convergence-markovian-penalization}
(\hat{X}^{n}, \hat{K}^{n}) \xrightarrow{d} (\hat{X}, \hat{K}), \quad \text{in distribution with respect to the uniform topology}
\end{equation}
where $(\hat{\Omega},\hat{\mathcal{F}},\hat{\mathbb{F}},\hat{\mathbb{Q}},\hat{W},\hat{X},\hat{K})$ denotes a weak solution to the reflected SDE
\begin{equation}\label{eq.Markovian1}
\mathrm{d}\hat{X}_t+ \mathrm{d}\hat{K}_t
= b^{1}(t,\hat{X}_t)\,\mathrm{d}t
+ \bar{\sigma}^{1}(t,\hat{X}_t)\,\mathrm{d}\hat{W}_t
 .
\end{equation} 
By Itô’s formula, we have 
\begin{equation}\label{first-prop}
\hat{P}
:= \hat{\mathbb{Q}} \circ \big(\hat{X}, \hat{K},
\hat{q}(t,\hat{X}_t)(\mathrm{d}u)\,\mathrm{d}t\big)^{-1}
\in \mathcal{R}_{ref(D)}(\mu).
\end{equation}  
Combining \eqref{convergence-markovian-penalization}, \eqref{samelaw1}, and \eqref{conv1}, we obtain, for every $t \in [0,T]$,
\begin{align*}
\hat{P}\circ (X_t)^{-1}
&= \hat{\mathbb{Q}}\circ (\hat{X}_t)^{-1}
= \lim_{n\to\infty} \mathbb{Q}^n \circ (\hat{X}^n_t)^{-1} \\
&= \lim_{n\to\infty} \mathbb{Q} \circ (X^n_t)^{-1}
= \mathbb{Q} \circ (X_t)^{-1}
= P \circ (X_t)^{-1}.
\end{align*}
Hence, assertion~(iii) follows.

To prove assertion~(ii), we invoke Skorokhod’s representation theorem. There exists a probability space
$(\tilde{\Omega},\tilde{\mathcal{F}},\tilde{\mathbb{Q}})$ and stochastic processes
$(\tilde{X}^n,\tilde{K}^n)$ and $(\tilde{X},\tilde{K})$ defined on it such that
\begin{equation}\label{samelaw2}
Law(\tilde{X}^n,\tilde{K}^n)=Law(\hat{X}^n,\hat{K}^n),
\qquad
Law(\tilde{X},\tilde{K})=Law(\hat{X},\hat{K}),
\end{equation}
and
\begin{equation}\label{conv2}
(\tilde{X}^n,\tilde{K}^n)\longrightarrow (\tilde{X},\tilde{K})
\quad \tilde{\mathbb{Q}}\text{-a.s.}
\end{equation}
Consequently,
\begin{align*}
J(\mu,\hat{P})
&= \mathbb{E}^{\hat{\mathbb{Q}}}\Bigg(
g(\hat{X}_T,\mu_T)
+ \int_0^T \mathrm{d}t \int_U q(t,\hat{X}_t)(\mathrm{d}u)\,
f(t,\hat{X}_t,\mu_t,u)
+ \int_0^T h(t,\hat{X}_t,\mu_t)\,\mathrm{d}\hat{K}_t
\Bigg)\\
&=\mathbb{E}^{\tilde{\mathbb{Q}}}\Bigg(
g(\tilde{X}_T,\mu_T)
+ \int_0^T \mathrm{d}t \int_U q(t,\tilde{X}_t)(\mathrm{d}u)\,
f(t,\tilde{X}_t,\mu_t,u)
+ \int_0^T h(t,\tilde{X}_t,\mu_t)\,\mathrm{d}\tilde{K}_t
\Bigg)  \\
&= \lim_{n\to\infty}
\mathbb{E}^{\tilde{\mathbb{Q}}}\Bigg(
g(\tilde{X}^n_T,\mu_T)
+ \int_0^T \mathrm{d}t \int_U q(t,\tilde{X}^n_t)(\mathrm{d}u)\,
f(t,\tilde{X}^n_t,\mu_t,u)
+ \int_0^T h(t,\tilde{X}^n_t,\mu_t)\,\mathrm{d}\tilde{K}^n_t
\Bigg) \\
&= \lim_{n\to\infty}
\mathbb{E}^{\mathbb{Q}^n}\Bigg(
g(\hat{X}^n_T,\mu_T)
+ \int_0^T \mathrm{d}t \int_U q(t,\hat{X}^n_t)(\mathrm{d}u)\,
f(t,\hat{X}^n_t,\mu_t,u)
+ \int_0^T h(t,\hat{X}^n_t,\mu_t)\,\mathrm{d}\hat{K}^n_t
\Bigg) \\
&= \lim_{n\to\infty}
\mathbb{E}^{\mathbb{Q}}\Bigg(
g(X^n_T,\mu_T)
+ \int_0^T \mathrm{d}t \int_U q(t,X^n_t)(\mathrm{d}u)\,
f(t,X^n_t,\mu_t,u)
+ \int_0^T h(t,X^n_t,\mu_t)\,\mathrm{d}K^n_t
\Bigg) \\
&= \lim_{n\to\infty}
\mathbb{E}^{\mathbb{Q}}\Bigg(
g(X^n_T,\mu_T)
+ \int_0^T \mathrm{d}t \int_U Q_t(\mathrm{d}u)\,
f(t,X^n_t,\mu_t,u)
+ \int_0^T h(t,X^n_t,\mu_t)\,\mathrm{d}K^n_t
\Bigg) \\
&= \mathbb{E}^{\mathbb{Q}}\Bigg(
g(X_T,\mu_T)
+ \int_0^T \mathrm{d}t \int_U Q_t(\mathrm{d}u)\,
f(t,X_t,\mu_t,u)
+ \int_0^T h(t,X_t,\mu_t)\,\mathrm{d}K_t
\Bigg) \\
&= J(\mu,P).
\end{align*}
For the first convergence, observe that the mapping $x \longmapsto \int_U f(t,x,\mu_t,u)\,q(t,x)(\mathrm{d}u)$ is only measurable. 
By invoking Krylov’s estimate, and arguing as in the proof of Lemma~2.5 in \cite{Slominski-convergence}, 
together with \eqref{Dom-measurable}, we can apply the dominated convergence theorem. 
The fourth equality follows from \eqref{samelaw2}, the fifth from \eqref{samelaw1}, 
and the final convergence is obtained by \eqref{conv1}.
\end{proof}
\begin{proof}[Proof of Theorem \ref{theorem:Markovian}]
Let $P \in \mathcal{R}_{ref(D)}^*(\mu)$, where $\mu = P \circ X^{-1}$, be a relaxed MFG solution given by Theorem \ref{MFGR}. Let $\hat{P}$ be the probability measure constructed in Lemma \ref{lemma:adj-Markovian}. The properties listed in Lemma \ref{lemma:adj-Markovian} allow us to build from $P$ a new relaxed MFG solution that is Markovian.\\
Point \textbf{ii)} shows that $\hat{P}$ also minimizes the expected cost among all admissible laws, then $\hat{P} \in \mathcal{R}_{ref(D)}^*(\mu)$.  
Point \textbf{iii)} guarantees that the marginals of $X$ are preserved under $\hat{P}$, meaning $\mu_t = \hat{P} \circ X_t^{-1}$ for every $t$. 
Point \textbf{i)} ensures that $\hat{P}$ is Markovian.
Thus $\hat{P}$ is a relaxed Markovian MFG solution.

Under Assumption \ref{ass:C}, the existence of a strict Markovian MFG solution follows from the same line of reasoning as in the proof of Theorem 3.7 in \cite{laker}. In particular, by applying the argument from the second part of that proof, we obtain measurable functions
$$
\alpha : [0,T] \times \mathbb{R}^d \to U,
\qquad
v : [0,T] \times \mathbb{R}^d \to \mathbb{R}^+,
$$
such that, for every t and almost every $\omega$,
$$
\int_{U} q(t,\hat{X}_t(\omega))(\mathrm{d}u)\,
\bigl(b, \sigma\sigma^{\top}, f\bigr)
\bigl(t,\hat{X}_t(\omega),\mu_t,u\bigr)
=
\bigl(b, \sigma\sigma^{\top}, f\bigr)
\bigl(t,\hat{X}_t(\omega),\mu_t,\alpha(t,\hat{X}_t(\omega))\bigr)
+
\bigl(0,0,v(t,\hat{X}_t(\omega))\bigr).
$$
This identity shows that the relaxed control can be replaced by the strict control $u = \alpha(t,\hat{X}_t)$, up to a nonnegative adjustment that affects only the running cost. As a result, the probability measure
$$
\hat{\mathbb{P}} := \hat{P} \circ \bigl(\hat{X},\, \hat{K},\, \mathrm{d}t\,\delta_{\alpha(t,\hat{X}_t)}(\mathrm{d}u)\bigr)^{-1}
$$
defines a strict Markovian MFG solution.
\end{proof}
\section{Approximation of Relaxed MFGs with Reflected Dynamics by Strict Controls with Penalized Dynamics} \label{sec-appro}
In this section, we show that relaxed MFG solutions with reflection, without assuming Assumption~\ref{ass:C}, can be approximated by strict controls whose dynamics are governed by penalized SDEs. 
The analysis relies on two main ingredients: the chattering lemma, which links relaxed and strict controls, and a penalization method that replaces reflected SDEs with standard SDEs augmented by a penalization term enforcing the boundary constraint.

Let’s begin with the following lemma, which includes two results: the first is the chattering lemma, and the second is based on the stability theorem for martingale measures. Both results are proved in \cite[p.~196]{meleard}.
\begin{lemma}\label{chattering lemma}
Let $q$ be a relaxed control, $M$ a martingale measure on $[0, T] \times U$ with intensity $q_t(\mathrm{d}u)\,\mathrm{d}t$, and $B$ a Brownian motion. All processes are defined on a filtered probability space $(\Omega', \mathcal{F}', \mathbb{F}', \mathbb{P}')$.
\begin{description}
    \item[i)] There exists a sequence of $\mathbb{F}'$-adapted processes $\alpha^n$ taking values in $U$, such that the sequence of random measures $\delta_{\alpha^n_t}(\mathrm{d}u)\, \mathrm{d}t$ converges to $q_t(\mathrm{d}u)\,\mathrm{d}t$ in $\mathcal{U}$, $\mathbb{P}'$-almost surely.
    \item[ii)] For any continuous and bounded function $\phi: [0, T] \times U \to \mathbb{R}$,
    \begin{equation*}
        \lim_{n \to \infty} \mathbb{E}^{\mathbb{P}'}\left( \left| \int_0^T \phi(s, \alpha^n_s)\, \mathrm{d}B_s - \int_0^T \int_U \phi(s, u)\,M(\mathrm{d}s, \mathrm{d}u) \right| \right) = 0.
    \end{equation*}
\end{description}
\end{lemma}
In this section, we assume that the initial condition is given by a deterministic point $x_0 \in \bar{D}$. 
The main result of this section is stated in the following theorem.
\begin{theorem} \label{stric-Non refle-equi relaxed}
Let $P \in \mathcal{R}^*_{\mathrm{ref}(D)}(\mu)$, where $\mu = P \circ X^{-1}$, be a relaxed MFG solution. 
Then there exists a sequence of strict controls $(P^n)_n \subset \mathcal{R}_n(\mu^n)$, with $\mu^n = P^n \circ X^{-1}$, associated with penalized dynamics, such that
\begin{description}
    \item[1)] $ \lim_{n \to \infty} W_{\Omega,2}(P^n, P) = 0 $,
    \item[2)] $\lim_{n \to \infty} J^n(\mu^n, P^n) =J(\mu, P)$.
\end{description}
\end{theorem}

\begin{proof}
Since $P \in \mathcal{R}^*_{ref(D)}(\mu)$ then by proposition \ref{rep martingale for Reflected dynamics}, there exists a filtered probability space $(\Omega', \mathcal{F}', \mathbb{F}', \mathbb{P}')$ supports a $d$-dimensional $ \mathcal{F}'_t$-adapted processes $(X, K)$ along with $ m $ orthogonal $\mathcal{F}'_t $-martingale measures $M= (M_1, \dots, M_m)$
on $U \times [0,T] $ with intensity $ Q_t(\mathrm{d}u)\mathrm{d}t $, satisfying $\mathbb{P}' \circ (X, K, Q)^{-1} = P$ and the following reflected SDE
\begin{equation*}
\mathrm{d}X_t+ \mathrm{d}K_t = \displaystyle\int_{U}^{} b(t,X_t,\mu_t, u)Q_t(\mathrm{d}u)\mathrm{d}t +\int_{U}^{}\sigma(t,X_t,\mu_t, u) M(\mathrm{d}u,\mathrm{d}t) .
\end{equation*}
According to Lemma~\ref{chattering lemma}, we may enlarge the probability space 
$(\Omega', \mathcal{F}', \mathbb{F}', \mathbb{P}')$ if necessary, and we keep the same notation for simplicity. 
On this extended space, there exist a Brownian motion $B$ and the 
$U$-valued processes $\alpha^n$ satisfying the two properties stated in 
Lemma~\ref{chattering lemma}.\\
For each $n \in \mathbb{N}$, let $X^{n}$ be the solution to the penalized SDE associated with the strict control $\alpha^{n}$, given by
\begin{equation*}
\mathrm{d}X^{n}_{t}
= b\big(t, X^{n}_{t}, \mu^{n}_{t}, \alpha^{n}_{t}\big)\, \mathrm{d}t
- n\big(X^{n}_{t} - \pi_{\bar{D}}(X^{n}_{t})\big)\, \mathrm{d}t
+ \sigma\big(t, X^{n}_{t}, \mu^{n}_{t}, \alpha^{n}_{t}\big)\, \mathrm{d}B_{t},
\end{equation*}
with initial condition $X^n_0 = x_0$, where $\mu^{n} = \mathbb{P}' \circ (X^{n})^{-1}$.\\
Applying Itô’s formula to $ X^{n}$ yields that the joint law of the state–control pair satisfies
\begin{equation*}
P^{n} := \mathbb{P}' \circ \left(X^{n},\, \delta_{\alpha^{n}_{t}}(\mathrm{d}u)\, \mathrm{d}t\right)^{-1}
\in \mathcal{R}_{n}(\mu^{n}).
\end{equation*}

Using the same arguments as in Lemmas~\ref{lemma2} and~\ref{lemma3}, together with Theorem~1.1 from~\cite[p.~354]{kushnerbook}, we conclude that the sequence of processes
$$
\Sigma^n := (X, X^n, M, \delta_{\alpha^n_t}(\mathrm{d}u)\, \mathrm{d}B_t, Q, \delta_{\alpha^n_t}(\mathrm{d}u)\, \mathrm{d}t,  K, -n\big(X^n_t - \pi_{\bar{D}}(X^n_t)\big)\, \mathrm{d}t),
$$
is tight.\\
By the Skorokhod representation theorem, there exists a probability space $(\hat{\Omega}, \hat{\mathcal{F}}, \hat{\mathbb{P}}) $, a sequence of random variables
$$
\Sigma'^n := (\tilde{X}^n, \hat{X}^n, \tilde{M}^n, \hat{M}^n, \tilde{Q}^n, \hat{Q}^n, \tilde{K}^n, \hat{K}^n),
$$
and a limit
$$
\Sigma' := (\tilde{X}, \hat{X}, \tilde{M}, \hat{M}, \tilde{Q}, \hat{Q}, \tilde{K}, \hat{K}),
$$
such that:
\begin{description}
    \item[P.1)] For each $ n \in \mathbb{N} $,
    $$
    Law(\Sigma^n) = Law(\Sigma'^n),
    $$
    i.e., the laws of the original and the copied processes coincide.
    
    \item[P.2)] There exists a subsequence $ \Sigma'^{n_k} $ (still denoted by $ \Sigma'^n $ for simplicity) such that
    $$
    \Sigma'^n \to \Sigma' \quad \text{almost surely under } \hat{\mathbb{P}}.
    $$
\end{description}
From property \textbf{P.1)}, it follows that the tuple $(\hat{X}^n, \hat{K}^n, \hat{M}^n, \hat{Q}^n) $ satisfies the penalized SDE:
\begin{equation}
\begin{cases}
\mathrm{d}\hat{X}^n_t + \mathrm{d}\hat{K}^n_t= \displaystyle\int_{U} b(t, \hat{X}^n_t, \mathcal{L}_{\hat{X}^n_t}, u)\, \hat{Q}^n_t(\mathrm{d}u)\, \mathrm{d}t 
+ \int_{U} \sigma(t, \hat{X}^n_t, \mathcal{L}_{\hat{X}^n_t}, u)\, \hat{M}^n(\mathrm{d}u, \mathrm{d}t) 
, \\
\hat{X}^n_0 = x_0.
\end{cases}
\end{equation}
Similarly, the processes \( (\tilde{X}^n, \tilde{K}^n, \tilde{M}^n, \tilde{Q}^n) \) satisfy:
\begin{equation}
\mathrm{d}\tilde{X}^n_t+ \mathrm{d}\tilde{K}^n_t = \displaystyle\int_{U} b(t, \tilde{X}^n_t, \mu_t, u)\, \tilde{Q}^n_t(\mathrm{d}u)\, \mathrm{d}t 
+ \int_{U} \sigma(t, \tilde{X}^n_t, \mu_t, u)\, \tilde{M}^n(\mathrm{d}u, \mathrm{d}t) 
.
\end{equation}
Passing to the limit and using property \textbf{P.2)} together with the arguments of \textbf{Step 1} in Subsection~\ref{subsection 4.2}, we obtain
$$
\mathrm{d}\hat{X}_t+\mathrm{d}\hat{K}_t
=\displaystyle\int_U b(t,\hat{X}_t,\mu_t,u)\,\hat{Q}_t(\mathrm{d}u)\,\mathrm{d}t
+\int_U \sigma(t,\hat{X}_t,\mu_t,u)\,\hat{M}(\mathrm{d}u,\mathrm{d}t)
.
$$
Applying Theorem~1.1 of \cite[page~354]{kushnerbook} yields the corresponding limit system for 
$(\tilde{X},\tilde{K},\tilde{M},\tilde{Q})$:
$$
\mathrm{d}\tilde{X}_t+\mathrm{d}\tilde{K}_t
=\displaystyle\int_U b(t,\tilde{X}_t,\mu_t,u)\,\tilde{Q}_t(\mathrm{d}u)\,\mathrm{d}t
+\int_U \sigma(t,\tilde{X}_t,\mu_t,u)\,\tilde{M}(\mathrm{d}u,\mathrm{d}t)
.
$$
By the chattering lemma (Lemma~\ref{chattering lemma}),
$$
Q^n:=\delta_{\alpha^{n}_{t}}(\mathrm{d}u)\, \mathrm{d}t\to Q \ \text{in }\mathcal{U},\quad \mathbb{P}'\text{-a.s.}\quad \text{and} 
\quad
M^n:=\delta_{\alpha^n_t}(\mathrm{d}u)\, \mathrm{d}B_t\to M \ \text{in }\mathcal{C}_{S'},\quad \mathbb{P}'\text{-a.s.}
$$
and therefore
$$
(Q^n,Q)\to(Q,Q)\ \text{in }\mathcal{U}^2,\quad \mathbb{P}'\text{-a.s.} \quad \text{and} 
\qquad
(M^n,M)\to(M,M)\ \text{in }\mathcal{C}_{S'}\times \mathcal{C}_{S'},\quad \mathbb{P}'\text{-a.s.}
$$
Using properties \textbf{P.1)} and \textbf{P.2)}, it follows that
$$
Law(Q^n, Q) = Law(\tilde{Q}^n, \hat{Q}^n) \quad \text{and} \quad Law(M^n, M) = Law(\tilde{M}^n, \hat{M}^n),
$$
and the sequence $(\tilde{Q}^n, \hat{Q}^n, \tilde{M}^n, \hat{M}^n)$ converges $\mathbb{P}'$-almost surely to $(\tilde{Q}, \hat{Q}, \tilde{M}, \hat{M})$.\\
By uniqueness of distributional limits, we obtain
$$
Law(\tilde{Q}, \hat{Q}) = Law(Q, Q) \quad \text{and} \quad Law(\tilde{M}, \hat{M}) = Law(M, M).
$$
Since $Law(Q, Q)$ and $Law(M, M)$ are supported on the diagonal, we conclude that
$$
\hat{\mathbb{P}}(\tilde{Q} = \hat{Q}) = 1 \quad \text{and} \quad \hat{\mathbb{P}}(\tilde{M} = \hat{M}) = 1.
$$
Since $(\tilde{X},\tilde{K})$ and $(\hat{X},\hat{K})$ solve the same reflected SDE with the same initial condition and are driven by $\hat{M}$ and $\hat{Q}$, pathwise uniqueness gives
$$
\hat{\mathbb{P}}\big( (\tilde{X},\tilde{K})=(\hat{X},\hat{K}) \big)=1.
$$
For point \textbf{1)}, we have
\begin{align*}
W_{\Omega, 2}(P^n,P)
&\le \mathbb{E}^{\mathbb{P}'}\left(
\sup_{t\le T}|X^n_t-X_t|^2
+ \sup_{t\le T}|K^n_t-K_t|^2
+ d_{\mathcal{U}}(Q^n,Q)^2 \right) \\
&= \mathbb{E}^{\hat{\mathbb{P}}}\left(
\sup_{t\le T}|\hat{X}^n_t-\tilde{X}^n_t|^2
+ \sup_{t\le T}|\hat{K}^n_t-\tilde{K}^n_t|^2
+ d_{\mathcal{U}}(\hat{Q}^n,\tilde{Q}^n)^2 \right) \\
&\xrightarrow[n\to\infty]{} 
\mathbb{E}^{\hat{\mathbb{P}}}\!\left(
\sup_{t\le T}|\hat{X}_t-\tilde{X}_t|^2
+ \sup_{t\le T}|\hat{K}_t-\tilde{K}_t|^2
+ d_{\mathcal{U}}(\hat{Q},\tilde{Q})^2 \right)=0.
\end{align*}
For point \textbf{2)},
\begin{align*}
|J(\mu,P)-J^n(\mu^n,&P^n)|
=\mathbb{E}^{\mathbb{P}'}\Bigg(
\Big|\int_0^T\int_U f(t,X_t,\mu_t,u)\,Q_t(\mathrm{d}u)\,\mathrm{d}t
 -\int_0^T\int_U f(t,X^n_t,\mu^n_t,u)\,Q^n_t(\mathrm{d}u)\,\mathrm{d}t \Big| \\
&\qquad + \Big|g(X_T,\mu_T)-g(X^n_T,\mu^n_T)\Big|
 + \Big|\int_0^T h(t, X_t, \mu_t)\,\mathrm{d}K_t - \int_0^T h(t, X^n_t,\mu^n_t )\,\mathrm{d}K^n_t\Big| \Bigg) \\
&= \mathbb{E}^{\hat{\mathbb{P}}}\Bigg(
\Big|\int_0^T\int_U f(t,\tilde{X}^n_t,\mu_t,u)\,\tilde{Q}^n_t(\mathrm{d}u)\,\mathrm{d}t
 -\int_0^T\int_U f(t,\hat{X}^n_t,\mu^n_t,u)\,\hat{Q}^n_t(\mathrm{d}u)\,\mathrm{d}t \Big| \\
&\qquad + \Big|g(\tilde{X}^n_T,\mu_T)-g(\hat{X}^n_T,\mu^n_T)\Big|
 + \Big|\int_0^T h(t,\tilde{X}^n_t,\mu_t )\,\mathrm{d}\tilde{K}^n_t - \int_0^T h(t,\hat{X}^n_t,\mu^n_t)\,\mathrm{d}\hat{K}^n_t\Big| \Bigg) \\
&\xrightarrow[n\to\infty]{} 
\mathbb{E}^{\hat{\mathbb{P}}}\Bigg(
\Big|\int_0^T\int_U f(t,\tilde{X}_t,\mu_t,u)\,\tilde{Q}_t(\mathrm{d}u)\,\mathrm{d}t
 -\int_0^T\int_U f(t,\hat{X}_t,\mu_t,u)\,\hat{Q}_t(\mathrm{d}u)\,\mathrm{d}t \Big| \\
&\qquad + \Big|g(\tilde{X}_T,\mu_T)-g(\hat{X}_T,\mu_T)\Big|
 + \Big|\int_0^T h(t,\tilde{X}_t,\mu_t)\,\mathrm{d}\tilde{K}_t - \int_0^T h(t, \hat{X}_t,\mu_t)\,\mathrm{d}\hat{K}_t\Big| \Bigg)=0,
\end{align*}
by continuity of $f$ and $g$ and the dominated convergence theorem.
\end{proof}

\bibliographystyle{abbrv}
\bibliography{bibliog}
\end{document}